\documentclass{amsart}

\input xy
\xyoption{all}

\usepackage{geometry}
\usepackage{amsmath}
\usepackage{amssymb}
\usepackage{cite}
\usepackage{amsthm}  % see geometry.pdf on how to lay out the page. There's lots.
\usepackage{tikz}
\usetikzlibrary{er,positioning}

\newtheorem{same}{This should never appear}[section]
\newtheorem{defin}[same]{Definition}

\newtheorem{remark}[same]{Remark}
\newtheorem{theorem}[same]{Theorem}
\newtheorem{example}[same]{Example}
\newtheorem{lemma}[same]{Lemma}
\newtheorem{fact}[same]{Fact}
\newtheorem{question}[same]{Question}
\newtheorem{cor}[same]{Corollary}

\newtheorem{nota}[same]{Notation}

\newtheorem{innercustomthm}{Theorem}
\newenvironment{customthm}[1]
  {\renewcommand\theinnercustomthm{#1}\innercustomthm}
  {\endinnercustomthm}

\newtheorem{defin*}{Definition}
\newtheorem*{theorem*}{Theorem}
\newbox\noforkbox \newdimen\forklinewidth
\setbox0\hbox{$\textstyle\smile$}
\setbox1\hbox to \wd0{\hfil\vrule width \forklinewidth depth-2pt
 height 10pt \hfil}
\setbox\noforkbox\hbox{\lower 2pt\box1\lower 2pt\box0\relax}
\def\unionstick{\mathop{\copy\noforkbox}\limits}

\def\nonfork_#1{\unionstick_{\textstyle #1}}

\setbox0\hbox{$\textstyle\smile$}
\setbox1\hbox to \wd0{\hfil{\sl /\/}\hfil}
\setbox2\hbox to \wd0{\hfil\vrule height 10pt depth -2pt width
              \forklinewidth\hfil}
\newbox\doesforkbox
\setbox\doesforkbox\hbox{\lower 2pt\box1 \lower 2pt\box2\lower2pt\box0\relax}
\def\nunionstick{\mathop{\copy\doesforkbox}\limits}

\def\fork_#1{\nunionstick_{\textstyle #1}}

\makeatletter
\newcommand{\skipitems}[1]{%
  \addtocounter{\@enumctr}{#1}%
}

\newcommand{\id}{\textrm{id}}

\newcommand{\ce}{\mathcal{C}}

\newcommand{\dnf}{\unionstick}

\newcommand{\s}{\mathfrak{s}}
\newcommand{\K}{\mathbf{K}}
\newcommand{\LS}{\operatorname{LS}}

\newcommand{\Ll}{\mathbb{L}}
\newcommand{\upp}{\upharpoonright}

\newcommand{\leap}[1]{\le_{#1}}
\newcommand{\ltap}[1]{<_{#1}}

\newcommand{\geap}[1]{\ge_{#1}}
\newcommand{\lta}{\ltap{\K}}
\newcommand{\lea}{\leap{\K}}

\newcommand{\gea}{\geap{\K}}

\newcommand{\gtp}{\mathbf{tp}}
\newcommand{\gS}{\mathbf{S}}

\newcommand{\Ii}{\mathbb{I}}

\title{Non-forking w-good frames}
\date{\today.} % delete this line to display the current date

\author{Marcos Mazari Armida}
\email{mmazaria@andrew.cmu.edu}
\address{Department of Mathematical Sciences \\ Carnegie Mellon University \\ Pittsburgh, Pennsylvania, USA}

\begin{document}

\maketitle

{\let\thefootnote\relax\footnote{{AMS 2010 Subject Classification: Primary 03C48. Secondary: 03C45, 03C55.
Key words and phrases. Abstract Elementary Classes; Good Frames; Tameness}}}  

%%%%%%%%%%%%%%%%%%%%%%%%%%%%%%%%%%%%%%%%%%%%%%%
\begin{abstract}

We introduce the notion of a w-good $\lambda$-frame which is a weakening of Shelah's notion of a good $\lambda$-frame. Existence of a w-good $\lambda$-frame implies existence of a model of size $\lambda^{++}$.  Tameness and amalgamation imply extension of a w-good $\lambda$-frame to larger models. As an application we show:
\begin{theorem}
Suppose $2^{\lambda}< 2^{\lambda^{+}} < 2^{\lambda^{++}}$ and $2^{\lambda^{+}} > \lambda^{++}$. If $\Ii(\K, \lambda) = \Ii(\K, \lambda^{+}) = 1 \leq \Ii(\K, \lambda^{++}) < 2^{\lambda^{++}}$ and $\K$ is $(\lambda, \lambda^+)$-tame, then $\K_{\lambda^{+++}} \neq \emptyset$.
\end{theorem}

The proof presented clarifies some of the details of the main theorem of \cite{sh576} and avoids using the heavy set-theoretic machinery of \cite[\S VII]{shelahaecbook} by replacing it with tameness.

\end{abstract}
%%%%%%%%%%%%%%%%%%%%%%%%%%%%%%%%%%%%%%%%%%%%%%%

\tableofcontents

%%%%%%%%%%%%%%%%%%%%%%%%%%%%%%%%%%%%%%%%%%%%%%%
\section{Introduction}
%%%%%%%%%%%%%%%%%%%%%%%%%%%%%%%%%%%%%%%%%%%%%%%

The central notion of Shelah's two volume book \cite{shelahaecbook} is that of a \emph{good $\lambda$-frame}, which is a forking-like notion for types of singletons in abstract elementary classes. It is crucial in transferring existence of models and categoricity to other cardinalities.

 Since it is hard to build good $\lambda$-frames, several weaker notions have been studied. Jarden and Shelah introduced  \emph{semi-good $\lambda$-frames} in \cite{jrsh875} and \emph{almost-good $\lambda$-frames} in \cite{jrsh940}. Vasey worked with \emph{$good^{-(S)} \lambda$-frames}  in \cite{vaseya} and with \emph{$good^- \lambda$-frames} in \cite{vaseyc}.\footnote{See Definition \ref{s-frame} for the definitions of all these notions and Diagram \ref{diagram} for their comparison in strength.} These notions have been particularly useful in deriving existence of models in larger cardinalities.

In this paper we introduce the notion of a \emph{w-good $\lambda$-frame} (see Definition \ref{extension-frame}), which is a weaker notion than all the ones mentioned above. A w-good $\lambda$-frame satisfies all the properties of a good $\lambda$-frame except that the density requirement is weakened and stability, symmetry and local character are not assumed.

In \cite[\S III.0]{shelahaecbook} Shelah introduced pre-$\lambda$-frames which are a weaker notion than that of a w-good $\lambda$-frame, but they are so weak that no interesting statement follows from their existence. The next diagram exhibits the relationship in strength between all the frames presented above. In the diagram, the source of an arrow is stronger than its target.\footnote{In Section 3.1 we present a more detailed discussion regarding the implications in the other direction.} \\

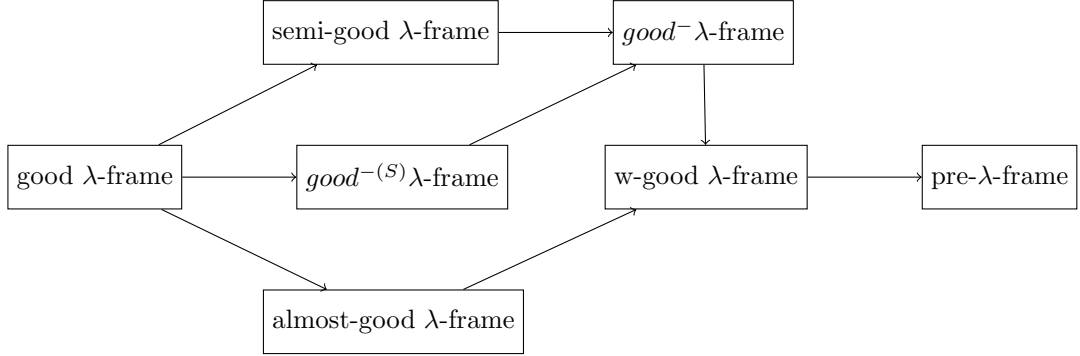
\begin{equation}\label{diagram}
\begin{tikzpicture}[auto,node distance=1.5cm]
  % Create an entity with ID node1, label "Fancy Node 1".
  % Default for children (ie. attributes) is to be a tree "growing up"
  % and having a distance of 3cm.
  %
  % 2 of these attributes do so, the 3rd's positioning is overridden.
  \node[entity] (node1) {good $\lambda$-frame}
  ;
\node[entity] (rel0) [above right = of node1] {semi-good $\lambda$-frame};

 \node[entity] (rel11) [right = of node1] {$good^{-(S)} \lambda$-frame};

\node[entity] (rel00) [right = of rel0] {$good^{-} \lambda$-frame};

  % Now place a relation (ID=rel1)
  \node[entity] (rel1) [below right = of node1] {almost-good $\lambda$-frame};

  % Now the 2nd entity (ID=rel2)
  \node[entity] (node2) [above right = of rel1]	{w-good $\lambda$-frame};

  \node[entity] (node3) [right = of node2]	{pre-$\lambda$-frame};

  % Draw an edge between rel1 and node1; rel1 and node2
  \path [->] (node1) edge node {} (rel1);
\path [->] (rel1) edge	 node {}	(node2);
\path [->](node1) edge node {} (rel0);
\path [->]  (rel0) edge	 node {}	(rel00);
\path[->] (rel11) edge node{} (rel00);
\path [->] (rel00) edge node {} (node2);
\path[->] (node2) edge node{} (node3);
\path[->](node1) edge node{} (rel11);

\end{tikzpicture}
\end{equation}\\

A w-good $\lambda$-frame is useful as it allows us to construct larger models. More precisely we show the following theorem which generalizes 
\cite[\S II.4.13.3]{shelahaecbook}, \cite[3.1.9]{jrsh875}  and \cite[8.9]{vaseya}:

%A w-good $\lambda$-frame is useful as it allows us to construct larger models. In \cite[\S II.4.13.3]{shelahaecbook} Shelah %showed that having a good $\lambda$-frame is enough to have a model of size $\lambda^{++}$, in \cite[3.1.9]{jrsh875} Jarden %and Shelah showed that the result still holds without assuming local character and stability while  in \cite[8.9]{vaseya} Vasey shows %that having a $good^{-(S)} \lambda$-frame is still enough. In this paper we show that a w-good $\lambda$-frame is actually %enough. More precisely we show the following generalization:

\begin{customthm}{3.18}
If $\s$ is a w-good $[\lambda, \mu)$-frame, then $\K_{\kappa} \neq \emptyset$ for all $\kappa \in [\lambda, \mu^+]$.  
\end{customthm}

 Under the hypothesis of tameness and the amalgamation property w-good $\lambda$-frames can be extended to larger models. The technique used to show this is similar to that of \cite[1.1]{extendingframes}, we only need to show that weak density and no maximal models transfer up.

\begin{customthm}{3.24}\label{ee}
Assume $\K$ is an AEC with the $[\lambda, \mu^+)$-amalgamation property. If $\s$ is a w-good $\lambda$-frame and $\K$ is $(\lambda, \mu)$-tame then $\s$  can be extended to a w-good $[\lambda, \mu^+)$-frame.
\end{customthm}

After presenting the above theorem, we apply the results obtained for w-good $\lambda$-frames to prove the following:

\begin{customthm}{4.2}\label{b-model}
Suppose $2^{\lambda}< 2^{\lambda^{+}} < 2^{\lambda^{++}}$ and $2^{\lambda^{+}} > \lambda^{++}$. If  $\Ii(\K,\lambda)= \Ii(\K,\lambda^+)=1\leq \Ii(\K,\lambda^{++}) < 2 ^{\lambda^{++}}$ and $\K$ is $(\lambda, \lambda^+)$-tame then $\K_{\lambda^{+++}} \neq \emptyset$.
\end{customthm}

The proof presented in the paper is an exposition of the  ideas displayed in \cite{sh576} with the following key feature. Using the assumption that $2^{\lambda^{+}} > \lambda^{++}$, that $\K$ is $(\lambda, \lambda^+)$-tame and the results obtained for w-good frames, we are able to avoid using the set-theoretic machinery developed in \cite[\S 3]{sh576} and in \cite[\S VII]{shelahaecbook} and used in Shelah's original proof. The set-theoretic machinery was initially developed by Shelah in a 20 pages section of \cite[\S 3]{sh576}, ten years later Shelah redid this section in a 200 pages chapter of his book  \cite[\S VII]{shelahaecbook}. In Shelah words ``Compared to \cite[\S 3]{sh576}, the present version [Chapter VII] is hopefully more transparent". This newer version was not refereed and we were still unable to verify Shelah's assertions.

Another interesting consequence of Theorem \ref{b-model} is that it gives a 200  pages shorter proof for the main theorem of \cite{sh576} (see Fact \ref{sh1}(1) below), with the extra hypothesis that $2^{\lambda^+} > \lambda^{++}$, in the case $\K$ is a universal class (see Definition \ref{universal}).

Lastly, we would like to point out that Theorem \ref{b-model} is not the best possible result in this direction, since 
the main theorem of \cite[\S VI.0.(2)]{shelahaecbook} (which is a revised version of \cite{sh576}) is the following.

\begin{fact}\label{sh1}
Suppose $2^{\lambda}< 2^{\lambda^{+}} < 2^{\lambda^{++}}$. If  $\Ii(\K,\lambda)=  \Ii(\K,\lambda^+)=1 \leq \Ii(\K,\lambda^{++}) < \mu_{unif}(\lambda^{++}, 2^{\lambda^+})$ \footnote{See \cite[\S VII.0.4]{shelahaecbook} for a definition of $\mu_{unif}$ and some of its properties.} then 
\begin{enumerate}
\item $\K_{\lambda^{+++}}\neq \emptyset$.
\item There is an almost-good $\lambda$-frame on $\K$.\footnote{Combining further results of Shelah, \cite[7.1]{vaseye} actually gets a good $\lambda$-frame and a good $\lambda^+$-frame.}
\end{enumerate}
\end{fact}

The paper is organized as follows. Section 2 presents necessary background. Section 3 introduces the notion of a w-good frame, shows that w-good frames imply the existence of larger models and shows how to extend w-good frames under tameness and amalgamation. Section 4 presents an exposition of the proof of the main theorem of \cite{sh576}, with the additional hypothesis that $2^{\lambda^+} > \lambda^{++}$ and $(\lambda, \lambda^+)$-tameness. The proof presented avoids using the set-theoretic machinery of \cite[\S VII]{shelahaecbook} by using $(\lambda, \lambda^+)$-tameness and the results of Section 3.

This paper was written while the author was working on a Ph.D. under the direction of Rami Grossberg at Carnegie Mellon University and I would like to thank Professor Grossberg for his guidance and assistance in my research in general and in this work in particular. I thank Hanif Cheung for helpful conversations. I thank Sebastien Vasey for very useful comments on an early version.

\section{Preliminaries}

We present the basic concepts of abstract elementary classes that we will need for the development of this paper. These are further studied in \cite[\S 4 - 8]{baldwinbook09} and  \cite[\S 2, \S 4.4]{ramibook}.

\subsection{Basic notions}

  First we will fix some notation.

\begin{nota}\
\begin{itemize}
\item Given $M \in \K$ we denote the universe of $M$ by $|M|$ and its cardinality by $\| M \|$.
\item Let $LS(\K) \leq \lambda < \mu$ such that $\lambda$ is an infinite cardinals and $\mu$ is an infinite cardinal or infinity. Let $[\lambda, \mu) = \{ \kappa \in card : \lambda \leq \kappa < \mu \}$. Given an abstract elementary class $\K$ and $[\lambda, \mu)$ an interval of cardinals, $\K_{[\lambda, \mu)} = \{  M \in \K : \| M\| \in [\lambda, \mu) \}$. In particular we let $\K_{\{\lambda \}} = \K_{[\lambda, \lambda^+)}=\K_{\lambda}$. 

\end{itemize}
\end{nota}

Let us recall the following three properties. They play a very important role in this paper, although not every AEC satisfies them.

\begin{defin}Let $LS(\K) \leq \lambda < \mu$ such that $\lambda$ is an infinite cardinals and $\mu$ is an infinite cardinal or infinity.
\begin{enumerate}
\item $\K_{[\lambda, \mu)}$ has the \emph{amalgamation property} (or $\K$ has the $[\lambda, \mu)$-amalgamation property): if for every $M, N, R \in \K_{[\lambda, \mu)}$ such that $M \lea N, R$, there are $f$ $\K$-embedding and $R^* \in \K$ such that $f: N \rightarrow_{M} R^*$ and $R \lea R^*$. 
\item  $\K_{[\lambda, \mu)}$ has the \emph{joint embedding property} (or $\K$ has the $[\lambda, \mu)$-joint embedding property): if for every $M, N \in \K_{[\lambda, \mu)}$, there are $f$ $\K$-embedding and $R^* \in \K$ such that $f: M \rightarrow R^*$ and $ N \lea R^*$.

\item  $\K_{[\lambda, \mu)}$ has \emph{no maximal models}: if for every $M \in \K_{[\lambda, \mu)}$, there is $M^* \in \K$ such that $M \lta M^*$.
\end{enumerate}

\end{defin}

The following fact was first proven in \cite{sh88}, but a more straightforward proof appears in \cite[4.3]{grossberg2002}.

\begin{fact}\label{AP}
Assume $2^{\lambda}< 2 ^{\lambda^{+}}$. Let $\K$ be an AEC.  If $\Ii(\K, \lambda)=1\leq \Ii(\K, \lambda^+) < 2^{\lambda^{+}}$ then $\K_{\lambda}$ has the amalgamation property. 
\end{fact}

\subsection{Galois types}
Let us begin by reviewing the concept of pre-type and some of its basic properties, pre-types will play a very important role in Section 4.

\begin{defin} \
\begin{enumerate}
\item The set of \emph{pre-types} is:
\[K_{\lambda}^{3}= \{ (a, M, N) : M \lea N, a \in |N| \backslash |M| \text{\, and \,} M, N \in \K_{\lambda} \}. \]
\item Given $(a_0, M_0, N_0), (a_1, M_1, N_1) \in K_{\lambda}^{3}$ we define:
\begin{enumerate}
\item $(a_0, M_0, N_0) \leq  (a_1, M_1, N_1)$ if and only if $M_0 \lea M_1$, $N_0 \lea N_1$ and $a_0 = a_1$.
\item $(a_0, M_0, N_0) < (a_1, M_1, N_1)$  if and only if  $(a_0, M_0, N_0) \leq  (a_1, M_1, N_1)$ and $M_0 \neq M_1$.
\end{enumerate}
\item Given $(a_0, M_0, N_0), (a_1, M_1, N_1) \in K_{\lambda}^{3}$ and $h: N_0 \rightarrow N_1$, we define $(a_0, M_0, N_0) \leq_{h}  (a_1, M_1, N_1)$ if and only if $h\upharpoonright_{M_0}: M_0 \rightarrow M_1$ is a $\K$-embedding,   $h: N_0 \rightarrow N_1$ is a $\K$-embedding and $h(a_0)=a_1$. 
\end{enumerate}
\end{defin}

We will also use the following property of pre-types, which is introduced in \cite[2.5]{sh576}. This will only be used in Section 4. 

\begin{defin}\label{d-reduced}
$(a_0, M_0, N_0) \in K_{\lambda}^{3}$ is \emph{reduced} if for any $(a_1, M_1, N_1) \in K_{\lambda}^{3}$ such that $(a_0, M_0, N_0) \leq (a_1, M_1, N_1)$ we have that $M_1 \cap N_0 = M_0$.
\end{defin}

The following appears as \cite[2.6(1)]{sh576} without a proof and it is proven in \cite[3.3.4]{jrsh875}.

\begin{fact}\label{reduce}
For every $(a_0, M_0, N_0) \in K_{\lambda}^{3}$ there is $(a_1, M_1, N_1) \in K_{\lambda}^3$ such that $(a_0, M_0, N_0) \leq (a_1, M_1, N_1)$ and $(a_1, M_1, N_1)$ is reduced. In that case, we say that reduced pre-types are dense in $K_{\lambda}^{3}$.
\end{fact} 

Let us recall the concept of Galois-type, this was introduced by Shelah in \cite{sh300}.

\begin{defin}\
\begin{enumerate}
\item Given $(a_0, M_0, N_0), (a_1, M_1, N_1) \in K_{\lambda}^{3}$ we say $(a_0, M_0, N_0)E_{at} (a_1, M_1, N_1)$ if $M:=M_0=M_1$ and there are $\{f_0, f_1\}$ and $N \in \K$ such that $f_l: N_l \rightarrow_{M} N$ for each $l \in \{0, 1\}$ and $f_0(a_0)=f_1(a_1)$. 
\item Let $E$ be the transitive closure of $E_{at}$.
\item Given $(a, M, N) \in K^3_{\lambda}$, we define the \emph{Galois-type} (also refer to as orbital type in the literature) as $\gtp(a/M, N) = [(a, M, N)]_E$. 
\item Given $M \in \K_{\lambda}$, let $\gS(M)= \{ \gtp(a/M, N) : M \lea N \in \K_\lambda \text{ and } a \in |N| \}$ and $\gS^{na}(M)= \{ \gtp(a/M, N) : (a, M, N) \in K_{\lambda}^3\}$. $\gS^{na}(M)$ is the set of nonalgebraic types.
\end{enumerate}
\end{defin}

The following is straightforward.

\begin{fact}
If $\K$ is an AEC and $\K_{\lambda}$ has the amalgamation property then $E_{at}$ is transitive. Hence $E_{at}= E$. 
\end{fact}

The concept of tameness was introduced by Grossberg and VanDieren in \cite{tamenessone}. We use this property to avoid using the set-theoretic machinery of \cite[\S VII]{shelahaecbook} mentioned in the introduction. The idea of using tameness instead of set-theoretic ideas traces back to \cite{tamenessone} and \cite{tamenesstwo}.

\begin{defin}
We say $\K$ is \emph{$(< \kappa)$-tame} if for any $M \in \K$ and $p \neq q \in \gS(M)$,  there is $N \lea M$ such that $\| N \| < \kappa$ and $p\upharpoonright_{N} \neq q\upharpoonright_{N}$. By $\K$ $\kappa$-tame we mean $(<\kappa^+)$-tame.  If we write $(\kappa, \lambda)$-tame we restrict to $M \in \K_{\lambda}$. 
\end{defin}

\section{w-good frames}

\subsection{Frames}

The concept of a good $\lambda$-frame is introduced in \cite[\S II.2, p. 259-263]{shelahaecbook}. We will follow the simplification and generalization given in \cite{vaseyc} and \cite{bova}. 

First let us recall the notion of a pre-frame.

\begin{defin}\label{p-frame}
Let $\lambda < \mu$ where $\lambda$ is an infinite cardinal and $\mu$ is an infinite cardinal or infinity. A \emph{pre-$[\lambda, \mu)$-frame} is a triple $( \K, \dnf, \gS^{bs})$ where the following properties hold:
\begin{enumerate}
\item $\K$ is an abstract elementary class with $\lambda \geq \LS(\K)$ and $\K_{\lambda} \neq \emptyset$.. 
\item $\gS^{bs} \subseteq \bigcup_{M \in \K_{[\lambda, \mu)}} \gS^{na}(M)$. Let $\gS^{bs}(M)= \gS(M) \cap \gS^{bs}$. 
 \item $\dnf$ is a relation on quadruples $(M_{0}, M_{1},  a , N)$, where $M_{0} \lea M_{1} \lea N$, $a \in N$ and $M_{0}, M_{1}, N \in \K_{[\lambda, \mu)}$. We write $ a \dnf_{M_{0}}^{N} M_1$ or $\gtp(a/M_1, N)$ does not fork over $M_{0}$ (which is well-defined by the next three properties).

\item \underline{Invariance:} If $f : N \cong N'$ and $ a \dnf_{M_{0}}^{N} M_1$ then $f(a) \dnf_{f[M_{0}]}^{N'} f[M_1]$. If $\gtp(a/M_{1}, N) \in \gS^{bs}(M_{1})$, then $\gtp(f(a)/ f[M_{1}], N') \in \gS^{bs}(f[M_{1}])$. 
\item \underline{Monotonicity:} If $ a \dnf_{M_{0}}^{N} M_1$ and $M_{0} \lea M_{0}' \lea M_{1}' \lea M_{1}\lea N' \lea N \lea N''$ with $N'' \in \K_{[\lambda, \mu)}$ and $a \in N'$, then $ a \dnf_{M_{0}'}^{N'} M_1'$ and $ a \dnf_{M_{0}'}^{N''} M_1'$. 
\item \underline{Nonforking types are basic:} If $ a \dnf_{M}^{N} M$, then $\gtp(a/M, N) \in \gS^{bs}(M)$. 

\end{enumerate}
\end{defin}

To simplify the comparison between the different kinds of frames we will introduce below, we recall the notion of good frame.

\begin{defin}\label{good-frame}
Let $\lambda < \mu$ where $\lambda$ is an infinite cardinal and $\mu$ is an infinite cardinal or infinity. A \emph{good $[\lambda, \mu)$-frame} is a triple $(\K, \dnf, \gS^{bs})$ where the following properties hold:
\begin{enumerate}
\item $( \K, \dnf, \gS^{bs})$  is a pre-$[\lambda, \mu)$-frame.
\item $\K_{[\lambda, \mu)}$ has amalgamation, joint embedding and no maximal models.
\item \underline{bs-Stability:} $|\gS^{bs} (M)| \leq ||M||$ for all $M\in \K_{[\lambda, \mu)}$.
\item \underline{Density of basic types:} If $M \lta N$ are both in $\K_{[\lambda, \mu)}$, then there is an $a \in |N|$ such that $\gtp(a/M , N) \in \gS^{bs}(M)$.
\item \underline{Existence of nonforking extension:} If $p \in \gS^{bs}(M)$ and $M \lea N$ with $N \in \K_{[\lambda, \mu)}$, then there is $q \in \gS^{bs}(N)$ that does not fork over $M$ and extends $p$.
\item \underline{Uniqueness:} If $M \lea N$ both in $\K_{[\lambda, \mu)}$, $p,q \in \gS(N)$ do not fork over $M$ and $p\upharpoonright _ M =  q \upharpoonright _ M$, then $p = q$. 
\item \underline{Symmetry:} If $ a_1 \dnf_{M_{0}}^{N} M_2$, $a_2 \in M_{2}$ and $\gtp(a_2/ M_{0}, N) \in \gS^{bs}(M_{0})$, then there are $M_1$ and $N'\gea N$ with $a_{1} \in M_1$ and $M_1, N \in \K_{[\lambda, \mu)}$ such that $ a_2 \dnf_{M_{0}}^{N'} M_1$.
\item \underline{Local character:} If $\delta < \mu$ is a limit ordinal, $\{ M_i : i < \delta \} \subseteq \K_{[\lambda, \mu)}$ is an increasing continuous chain and $p \in \gS^{bs}(M_\delta)$, then there is an $i < \delta$ such that $p$ does not fork over $M_i$. 
\item \underline{Continuity:} If $\delta < \mu$ is a limit ordinal, $\{ M_i : i < \delta \} \subseteq \K_{[\lambda, \mu)}$ is an increasing continuous chain, $\{p_{i} : i < \delta \}$ with $p_i \in \gS^{bs}(M_i)$ and for $i < j < \delta$ implies that $p_i = p_j \upharpoonright _ {M_i}$ and $p \in \gS^{na} (M_\delta)$ is an upper bound for $\{p_{i} : i < \delta \}$, then $p \in \gS^{bs} (M_\delta)$ . Moreover, if each $p_i$ does not fork over $M_0$, then neither does $p$.
\item \underline{Transitivity:} If $M_0 \leq M_1 \leq M_2$ with $M_0, M_1, M_2 \in \K_{[\lambda, \mu)}$, $p \in \gS(M_2)$ does not fork over $M_1$ and  $p\upharpoonright_ {M_1}$ does not fork over $M_0$, then $p$ does not fork over $M_0$. 
\end{enumerate}
\end{defin}

Recall the following notation which was introduced in \cite{vaseyc}.

\begin{nota}\label{notation}
Given $\Ll$ a list of properties a $good^{-\Ll} \lambda$-frame is a pre-$\lambda$-frame that satisfies all the properties of a good $\lambda$-frame except possibly the properties listed in $\Ll$. We abbreviate stability by St, density by D, symmetry by S and local character by Lc.

\end{nota}

In \cite{jrsh940} Jarden and Shelah introduced the following weakening of local character.

\begin{defin}
A $( \K, \dnf, \gS^{bs})$  pre-$\lambda$-frame satisfies \emph{weak local character} if there is a 2-ary relation $\leq^{\ast}$ in $\K_{\lambda}$ such that:
\begin{itemize}
\item  If $M \leq^\ast N$ both in $\K_\lambda$ then $M \lea N$.
\item For every $M \in \K_\lambda$, there is $N \in \K_\lambda$ such that $M <^{\ast} N$.
\item If $M \leq^\ast N \lea R$ all in $\K_\lambda$, then $M \leq^\ast R$.
\item If $\delta < \lambda^+$ is a limit ordinal and $\{ M_i : i < \delta +1  \} \subseteq \K_{\lambda}$ is an $\leq^\ast$-increasing continuous chain, then there are $a \in |M_{\delta +1}| \backslash |M_\delta|$ and $\alpha < \delta$ such that $\gtp(a/M_\delta, M_{\delta +1})\in \gS^{bs}(M_\delta)$ and does not fork over $M_\alpha$. 
\end{itemize}
\end{defin}

Since good $\lambda$-frames were introduced several weaker notions have been studied. In the definition below we recall all of them and write in parenthesis the paper in which they were introduced.

\begin{defin}\label{s-frame} \
\begin{enumerate}
\item (\cite{jrsh875})  A \emph{semi-good $\lambda$-frame} is a $good^{-(St)} \lambda$-frame with the additional property that for any $M \in \K_{[\lambda, \mu)} (|\gS^{bs}(M)| \leq \| M \|^{+})$. 
\item (\cite{jrsh940}) An \emph{almost-good $\lambda$-frame} is a $good^{-(Lc)} \lambda$-frame with the additional property that it satisfies weak local character.
\item (\cite{vaseyc}) A \emph{$good^{-(S)} \lambda$-frame} is a good $\lambda$-frame without symmetry.
\item (\cite{vaseya}) A \emph{$good^{-} \lambda$-frame} is a $good^{-(St, S)} \lambda$-frame.

\end{enumerate}
Diagram \ref{diagram} shows how they compare to one another. 
\end{defin}

Before introducing the notion of a w-good frame, we will introduce a notion of weak density.

\begin{defin}\label{extension-frame}
 A $( \K, \dnf, \gS^{bs})$  pre-$[\lambda, \mu)$-frame has \emph{weak density for basic types} when: if $M \in \K_{\lambda}$ and $M  \lta N \in \K_{[\lambda, \mu)}$ then there are $a \in |N| \backslash |M|$ and $M' \lta N'$ both in $\K_{[ \lambda, \mu)}$ such that $a \in |N'| \backslash |M'|$, $\gtp(a / M', N') \in \gS^{bs}(M)$ and $(a, M, N) \leq (a, M', N')$. 
\end{defin}

Observe that if a pre-frame has density for basic types then it has weak density for basic types. We do not know if under the other axioms of a good $\lambda$-frames the conditions are equivalent (but we suspect it is not the case).\footnote{ Shelah shows in \cite[\S VI.7.4]{shelahaecbook} that under additional hypothesis weak density implies density.}

We are ready to introduce the notion of a w-good frame.

\begin{defin}Let $\lambda < \mu$ where $\lambda$ is an infinite cardinal and $\mu$ is an infinite cardinal or infinity. 
A \emph{w-good $[\lambda, \mu)$-frame} is triple $(\K, \dnf, \gS^{bs})$ where the following properties hold:

\begin{enumerate}
    \item $(\K, \dnf, \gS^{bs})$  is a pre-$[\lambda, \mu)$-frame
    \item $\K_{[\lambda, \mu)}$ has amalgamation, joint embedding and no maximal models.
    \skipitems{2}
\item[$(4)^-$]  \underline{Weak density}
    \item \underline{Existence of nonforking extension}
\item \underline{Uniqueness}
    \skipitems{2}
    \item \underline{Continuity}
  \end{enumerate}

Using the notation introduced in \ref{notation}, a w-good $[\lambda, \mu)$-frame is a $good^{-(St,D, S, Lc)}\lambda$-frame with the additional property that it satisfies weak density. 
\end{defin}

\begin{remark}\
\begin{itemize}

\item As in \cite[\S II.2.18]{shelahaecbook} one can show that in a w-good frame transitivity of non-forking holds. 
\item As we can see by comparing Definition \ref{s-frame} and Definition \ref{extension-frame}, a w-good frame is weaker than all the notions presented in Definition \ref{s-frame}. 
\end{itemize}

\end{remark}

It is natural to ask which of the notions introduced in Definition \ref{s-frame} and Definition \ref{extension-frame} are strictly stronger. In \cite[\S 2.2]{jrsh875} Jarden and Shelah showed that good $\lambda$-frames are strictly stronger than semi-good $\lambda$-frames. Adapting an example of \cite[\S 2]{jrsh875}, we show that w-good $\lambda$-frames are strictly stronger than pre-$\lambda$-frames.

\begin{example}
Let $L(\K)= \{ <\}$, where $<$ is a binary relation, and $\K = ( Mod(T_{LO}), \subseteq)$, where $T_{LO}$ is the first order theory of linear orders. Let $\s = (\K, \gS^{na}, \dnf)$ where for $M_0, M_1, N \in \K_\lambda$: $ a \dnf_{M_{0}}^{N} M_1$ if and only if $M_0 \subseteq M_1 \subseteq N$ and $a \in |N| \backslash |M_1|$. It is trivial to check that $\s$ is a pre-$\lambda$-frame. Moreover, the uniqueness property fails so $\s$ is not a w-good $\lambda$-frame.

\end{example}

The following example shows that $good^{-} \lambda$-frames are strictly stronger than w-good $\lambda$-frames. This example appears in a different context in other papers (\cite[II.6.4]{shelahaecbook}, \cite[6.6]{adler} and \cite[4.15]{canonicity}).
\begin{example}\label{gra}
Let $L(\K) =\{ E\}$, where $E$ is a binary relation, and $\K= (Mod(T_{ind}), \preceq)$, where $T_{ind}$ is the first-order theory of the random graph.  Let $\s = (\K, \gS^{na}, \dnf)$ where for $M_0, M_1, N \in \K_\lambda$: $ a \dnf_{M_{0}}^{N} M_1$ if and only if $M_0 \preceq M_1 \preceq N$, $a \in |N| \backslash |M_1|$ and there are no edges between $a$ and $|M_1| \backslash |M_0|$.  

It is easy to check that $\s$ is a w-good $\lambda$-frame, we show that $\s$ does not have local character. Build $\{M_i : i < \omega \} \subseteq \K_{\lambda}$ strictly increasing and continuous. Let $M_\omega = \bigcup_{i< \omega} M_i$ and let $N \in \K_\lambda$ such that $M_\omega \preceq N$ and there is $a\in |N| \backslash |M_\omega|$ such that for every $b \in M_\omega$ there is an edge between $a$ and $b$. Observe that $tp(a/M_\omega, N)$ does not fork over $M_\omega$, but for any $i < \omega$  $tp(a/M_\omega, N)$ forks over $M_i$.

Therefore, $\s$ is a w-good $\lambda$-frame and it is not a $good^{-} \lambda$-frame.
\end{example}

 Adapting another example of \cite[\S 2]{jrsh875}, we show that semi-good $\lambda$-frames and $good^{-(S)} \lambda$-frames are strictly stronger than $good^{-} \lambda$-frames. Moreover, the example also exhibits that almost-good $\lambda$-frames are strictly stronger than w-good $\lambda$-frames.
\begin{example}
Suppose that $2^{\lambda} \geq \lambda^{++}$. Let $L(\K)= \{R_\alpha : \alpha < \lambda \}$, where each $R_{\alpha}$ is a unary predicate, and $\K = ( L(\K)\text{-structures }, \subseteq)$. Let $\s = (\K, \gS^{na}, \dnf)$ where for $M_0, M_1, N \in \K_\lambda$: $ a \dnf_{M_{0}}^{N} M_1$ if and only if $M_0 \subseteq M_1 \subseteq N$ and $a \in |N| \backslash |M_1|$.

$\K_\lambda$ satisfies amalgamation, joint embedding  and no maximal models. Moreover, for every $M_0, M_1, N_0, N_1 \in \K_{\lambda}$, $a_0 \in |N_0| \backslash |M_1|$ and $a_1 \in  |N_1| \backslash |M_1|$ it follows that:
\[ \gtp(a_0/ M_0, N_0) = \gtp(a_1/ M_1, N_1) \text{ if and only if } \{ \alpha < \lambda : a_0 \in R_\alpha^{N_0}\} = \{ \alpha < \lambda : a_1 \in R_\alpha^{N_1}\} .\]

Using this property it is easy to show that all the conditions of a  $good^{-} \lambda$-frame are satisfied and that for any $M \in\K_\lambda (|\gS^{bs}(M)| = 2^{\lambda})$. Since $2^{\lambda}\geq \lambda^{++}$ it follows that $\s$ is neither a semi-good $\lambda$-frame or a $good^{-(S)} \lambda$-frame. Observe that the hypothesis that  $2^{\lambda} \geq \lambda^{++}$ is only used to show that $\s$ is not a semi-good $\lambda$-frame.

For the moreover part, it is clear that $\s$ is a w-good $\lambda$-frames, but not an almost-good $\lambda$-frame. 
\end{example} 

Below we revise the diagram of the introduction, we write ``s" above those arrows for which it is known that the source frame is strictly stronger than the target frame and we write ``s*" above those arrows for which it is known that the source frame is strictly stronger than the target frame but under some set-theoretic hypothesis.\\
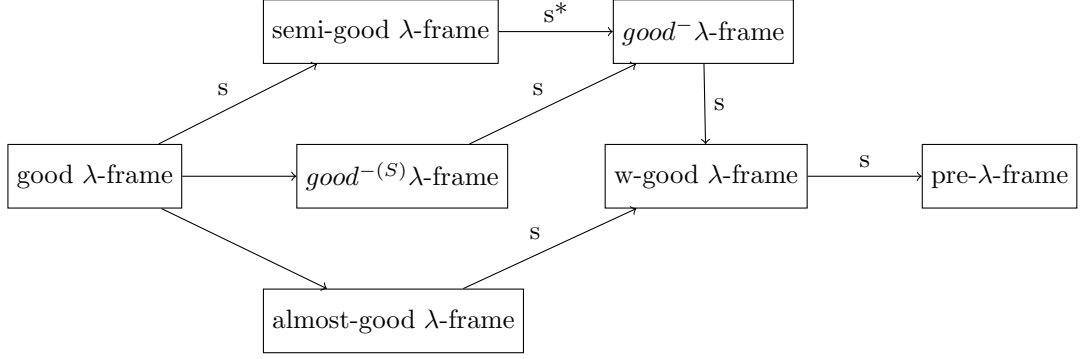
 \begin{equation}
\begin{tikzpicture}[auto,node distance=1.5cm]
  % Create an entity with ID node1, label "Fancy Node 1".
  % Default for children (ie. attributes) is to be a tree "growing up"
  % and having a distance of 3cm.
  %
  % 2 of these attributes do so, the 3rd's positioning is overridden.
  \node[entity] (node1) {good $\lambda$-frame}
  ;
\node[entity] (rel0) [above right = of node1] {semi-good $\lambda$-frame};

 \node[entity] (rel11) [right = of node1] {$good^{-(S)} \lambda$-frame};

\node[entity] (rel00) [right = of rel0] {$good^{-} \lambda$-frame};

  % Now place a relation (ID=rel1)
  \node[entity] (rel1) [below right = of node1] {almost-good $\lambda$-frame};

  % Now the 2nd entity (ID=rel2)
  \node[entity] (node2) [above right = of rel1]	{w-good $\lambda$-frame};

  \node[entity] (node3) [right = of node2]	{pre-$\lambda$-frame};

  % Draw an edge between rel1 and node1; rel1 and node2
  \path [->] (node1) edge node {} (rel1);
\path [->] (rel1) edge	 node {s}	(node2);
\path [->](node1) edge node {s} (rel0);
\path [->]  (rel0) edge	 node {s*}	(rel00);
\path[->] (rel11) edge node{s} (rel00);
\path [->] (rel00) edge node {s} (node2);
\path[->] (node2) edge node{s} (node3);
\path[->](node1) edge node{} (rel11);

\end{tikzpicture}
\end{equation}\\

\begin{question}
Are any of the notions introduced above the same? Are all the notions introduced above the same under some additional hypothesis on $\K$?
\end{question}

\begin{question}
Let $T$ be a first-order theory. It is easy to show that if $T$ is $\lambda$-stable then $T$ has a w-good $\lambda$-frame. Example \ref{gra} shows that simple theories might have a w-good $\lambda$-frame. So the question is: under what hypothesis does $T$ have a w-good  $\lambda$-frame?

Another interesting question in this neighborhood is the following: is there a w-good $\lambda$-frame on a $\lambda$-stable theory $T$ different from first-order nonforking?
\end{question}

\subsection{Inside a w-good $[\lambda, \mu)$-frame}
Let us recall the definition of a coherent sequence of types. It appears in print in \cite{extendingframes}.

\begin{defin}
Given $\{ M_i : i < \delta \}$ an increasing continuous chain and $\{ p_i \in \gS^{na}(M_i) : i < \delta \}$ an increasing sequence of types, the sequence is a \emph{coherent sequence of types} if and only if there are $\{ (a_i, N_i) : i < \delta \}$ and $\{ f_{j,i} : j < i < \delta \}$ such that:
\begin{enumerate}
\item $f_{j,i}: N_j \rightarrow N_i$.
\item For all $k < j < i$, we have $f_{k,i} = f_{j,i} \circ f_{k,j}$.
\item $\gtp(a_i/ M_i, N_i) = p_i$.
\item $f_{j,i}\upharpoonright_{M_j}=\id_{M_j}$.
\item $f_{j,i}(a_j)= a_i$.
\end{enumerate} 
\end{defin}

The following lemma is straightforward but due to its importance in what follows we will sketch the proof.

\begin{lemma}
Let $\{ p_i \in \gS^{na}(M_i) : i < \delta \}$ a coherent sequence of types, then there is $p \in \gS^{na}(M_{\delta})$ upper bound for the sequence of types, i.e., for every $i< \delta ( p \text{ extends } p_i )$. 
\end{lemma}
\begin{proof}
 Let $( N , \{f_{i} : N_{i} \rightarrow N : i < \delta \})$ be the direct limit of the sequence such that $M_{\delta} \lea N$ and $f_i\upp_{M_i}=\id_{M_i}$. Let $a := f_{0}(a_0)$ and $p:= \gtp(a /M_{\delta}, N)$. Observe that $\gtp(a /M_{\delta}, N) \in \gS^{na}(M_{\delta})$, if $a \in M_{\delta}$ then there is $i < \delta$ such that $a \in M_i$, then using that $f_i \circ f_{0,i} = f_0$, $f_{0,i}(a_0)= a_i$ and $f_i\upp_{M_i}= \id_{M_i}$ it follows that $a_i \in M_i$, which contradicts the fact that $p_i$ is nonalgebraic. It is easy to show that $\gtp(a /M_{\delta}, N)$ is an upper bound for the sequence of types.
\end{proof}

\begin{lemma}\label{coherent}
Let $\s$ be a w-good $[\lambda, \mu)$-frame. Let $\{ M_i \in \K_{[\lambda, \mu)} : i < \delta \}$ an increasing continuous chain such that $\delta \leq \mu$. If $\{ p_i \in \gS^{bs}(M_i) : i < \delta \}$ is an increasing sequence of types such that $p_i$ does not fork over $M_0$ for every $i < \delta$, then $\{ p_i : i < \delta \}$ is a coherent sequence of types. Moreover, there is $p_{\delta} \in \gS^{na}(M_{\delta})$ extending all the $p_i$. If $\delta < \mu$ then  $p_{\delta}\in \gS^{bs}(M_\delta)$ does not fork over $M_0$. 
\end{lemma}
\begin{proof}
The exact same proof of \cite[5.2]{extendingframes} works, since the only properties of good frames that are used are amalgamation, uniqueness and continuity.
\end{proof}

\begin{lemma}\label{model-lambda}
If $\s$ is a w-good $[\lambda, \mu)$-frame without the assumption that $\K_{(\lambda,\mu)}$ has no maximal models, then $\K_{[\lambda, \mu]}$ has no maximal models. 
\end{lemma}
\begin{proof}
We show that for every $\kappa \in [\lambda, \mu]$ $\K_{\kappa}$ has no maximal models. The case when $\kappa = \lambda$ follows directly from the definition of w-good $[\lambda, \mu)$-frame and the assumption. So suppose $\kappa \in (\lambda, \mu]$ and $M \in \K_{\kappa}$ is a maximal model. Let $R \lta S \lea M$ such that $R, S \in \K_{\lambda}$. By weak density, $[\lambda, \mu)$-amalgamation property and using the fact that $M$ is maximal, there are $R' \lta S' \lea M$ both in $\K_{\lambda}$ and $a \in |S| \backslash |R|$ such that $\gtp(a/ R', S') \in \gS^{bs}(R')$.

We build $\{ M_i : i < \kappa \} \subseteq \K_{< \kappa}$ an increasing and continuous resolution of $M$ such that $M_0:= R'$. We build $\{ p_i : i < \kappa \}$ such that:
\begin{enumerate}
\item $p_0 = \gtp(a/ R', S')$.
\item For all $i < \kappa$, $p_i \in \gS^{bs}(M_i)$.
\item For all $i < \kappa$, $p_i$ does not fork over $M_0$.
\item If $j < i$ then $p_j \leq p_i$.
\end{enumerate}

\underline{Enough:} By Lemma \ref{coherent} there is $p \in \gS^{na}(\bigcup_{i < \kappa} M_i)$. Observe that $\bigcup_{i < \kappa} M_i = M$ and since the type is nonalgebraic there is $N \in \K_{\kappa}$ and $a \in |N| \backslash |M|$ such that $p = \gtp(a/M,N)$. Hence $M \lta N$, this contradicts the fact that $M$ is maximal.

\underline{Construction:} The base step is (1) and if $i$ is limit we apply Lemma \ref{coherent}. So the only interesting case is when $i= j+1$. By construction we have $p _j \in \gS^{bs}(M_j)$ that does not fork over $M_0$. Since $M_j \lta M_{j+1}$ and both models are in $\K_{[\lambda,\kappa)}$, by the extension property there is $p_{j+1} \in \gS^{bs}(M_{j+1})$ such that $p_j \leq p_{j+1}$ and $p_{j+1}$ does not fork over $M_j$. Then by transitivity $p_{j+1}$ does not fork over $M_0$ 

\end{proof}

\begin{theorem}\label{model}
If $\s$ is a w-good $[\lambda, \mu)$-frame, then $\K_{\kappa} \neq \emptyset$ for all $\kappa \in [\lambda, \mu^+]$.  
\end{theorem}
\begin{proof}
It follows from the fact that $\K_{\lambda} \neq \emptyset$ and Lemma \ref{model-lambda}.
\end{proof}

 The following corollary has a long history. First, Shelah proved it for good $\lambda$-frames in \cite[\S II.4.13]{shelahaecbook}, then Jarden and Shelah proved it for $good^{-(St,Lc)} \lambda$-frames\footnote{It is clear that a $good^{-(St,Lc)} \lambda$-frame is stronger than a w-good $\lambda$-frame. It is suspected that symmetry does not follow from the other axioms of a good $\lambda$-frame, so we suspect that  $good^{-(St,Lc)} \lambda$-frames are strictly stronger than w-good $\lambda$-frames. The reason we do not mention $good^{-(St,Lc)} \lambda$-frames until this point is because they are simply a technical tool develop in \cite{jrsh875} to encompass both semi-good frames and almost-good frames.} in \cite[3.1.9]{jrsh875} . Later Vasey proved it for $good^{-(S)} \lambda$-frames in \cite[8.9]{vaseya}. Below we prove it for w-good $\lambda$-frames.

\begin{cor}\label{model-lambda++}
If $\s$ is a w-good $\lambda$-frame, then $\K_{\lambda^+} \neq \emptyset$ and $\K_{\lambda^{++}} \neq \emptyset$. 
\end{cor}
\begin{proof}
Observe that $\s$ is a w-good $[\lambda, \lambda^+)$-frame and use Theorem \ref{model}.
\end{proof}

\subsection{Extending w-good $\lambda$-frames}

Similarly to \cite{extendingframes}\footnote{\cite{extendingframes} uses tameness for 2-types to extend symmetry,  in \cite[6.9]{bova} it was established that tameness for 1-types is sufficient. Observe that in this paper the results of  \cite{extendingframes} are enough since symmetry is not assumed.}, one can show that under the amalgamation property and tameness one can extend a w-good $\lambda$-frame to a w-good $[\lambda, \infty)$-frame. We will only sketch the proof since all the proofs of \cite{extendingframes} work for our weaker setting, except the proof of weak density and of no maximal models.

The following definition is a local version of $\geq \s$ which appears in \cite[\S II.2]{shelahaecbook} for good $\lambda$-frames. 

\begin{defin}\label{def-ext} Let $LS(\K)\leq \lambda < \mu$ where $\lambda$ is an infinite cardinal and $\mu$ is an infinite cardinal or infinity.
Given $\s$ a w-good $\lambda$-frame we define:
\begin{itemize}
\item $K_{\s_{[\lambda, \mu)}}^{3, bs} = \{ (a, M, N) \in K^3_{[\lambda, \mu)} : \text{there is \;} M' \lea M \text{\; in \,} \K_{\lambda} \text{\; such that: if \,} M'' \in \K_{\lambda} \text{\; with \;} M' \lea M'' \lea M \text{\; , then \,} \gtp(a/ M'', N) \text{\, does not fork over \,} M' \}$.
\item $\gS_{\s_{[\lambda, \mu)}}^{bs} = \{ p \in \gS(M) : \text{\, for some/every \,} (a, M, N) \in K_{\s_{[\lambda, \mu)}}^{3, bs} (p = \gtp(a/M, N) \}$.
\item Given $M_0 \lea M_1 \lea N$ all in $\K_{[\lambda, \mu)}$ and $a \in |N| \backslash |M_1|$: $ a \dnf_{M_{0}}^{N} M_1$ if and only if there is $R_0 \lea M_0$ in $\K_{\lambda}$ such that for any $R_1, S \in \K_{\lambda}$ with $R_0 \lea R_1 \lea M_1$ and $R_1 \cup \{a \} \subseteq S \leq N $ it holds that $a \dnf_{R_{0}}^{S} R_1$.
\end{itemize}
Define $\s_{[\lambda, \mu)} = (\K, \gS_{\s_{[\lambda, \mu)}}^{bs}, \dnf_{\s_{[\lambda, \mu)}})$.
\end{defin}

The following is already proven for good $\lambda$-frames in \cite{extendingframes}.

\begin{lemma}\label{will-ext}
Assume $\K$ is an AEC with the $[\lambda, \mu^+)$-amalgamation property. If $\s$ is a w-good $\lambda$-frame and $\K$ is $(\lambda, \mu)$-tame then $\s_{[\lambda, \mu^+)}$ is a pre-$[\lambda, \mu^+)$-frame that satisfies the amalgamation property, the joint embedding property, existence of non-forking extension, uniqueness and continuity.
\end{lemma}
\begin{proof}
It is trivial to show that $\s_{[\lambda, \mu^+)}$ is a pre-$[\lambda, \mu^+)$-frame. We have the amalgamation property by hypothesis and the joint embedding property follows from the amalgamation property and the fact that we have the joint embedding property in $\K_{\lambda}$. The existence of non-forking extension is \cite[5.3]{extendingframes}, the uniqueness property is \cite[3.2]{extendingframes} and continuity is \cite[\S II 2.11(6)]{shelahaecbook}. 
\end{proof}

Therefore we only need to prove that weak density and no maximal models transfer up.

\begin{lemma}\label{b-wd}
Assume $\K$ is an AEC with the $[\lambda, \mu^+)$-amalgamation property. If $\s$ is a pre-$\lambda$-frame that has weak density then $\s_{[\lambda, \mu^+)}$ has weak density.
\end{lemma}
\begin{proof}
Let $M \lta N$ such that $M \in \K_{\lambda}$. If $N \in \K_{\lambda}$ then it follows directly from the fact that $\s$ satisfies the weak density property. So let us do the case when $\| N \| > \lambda$.

Apply LST axiom to get $N_0 \in \K_{\lambda}$ such that $M \lta N_0 \lea N$. By weak density in $\s$ there are $a \in |N_0| \backslash |M|$ and $M' \lta N'_0$ both in $\K_{\lambda}$ such that $a \in |N'_0| \backslash |M'|$, $\gtp(a / M', N'_0) \in \gS^{bs}(M')$ and $(a, M, N_0) \leq (a, M', N'_0)$.  By the amalgamation property there are $f$ and $N' \in \K_{[\lambda, \mu]}$ such that the following diagram commutes:
\[
 \xymatrix{\ar @{} [dr] N_0'  \ar[r]^{f}  & N' \\
N_{0} \ar[u]^{\id} \ar[r]^{\id} & N \ar[u]_{\id}
 }
\]

Observe that $f(a)=a$, $a \in |N'| \backslash |f[M']|$, $(a, M, N) \leq (a, f[M'], N')$ and $\gtp(a/ f[M'], N') \in \gS_{\s_{[\lambda, \mu^+)}}^{bs}(f[M'])$.
\end{proof}

 The reason we can not simply quote \cite[7.1]{extendingframes} to transfer up no maximal models is because Boney's proof uses symmetry, which we are not assuming. 

\begin{lemma}\label{nmm}
Assume $\K$ is an AEC with the $[\lambda, \mu^+)$-amalgamation property. If $\s$ is a w-good $\lambda$-frame and $\K$ is $(\lambda, \mu)$-tame then $\s_{[\lambda, \mu^+)}$ has no maximal models.
\end{lemma}
\begin{proof}
By Lemma \ref{will-ext} and Lemma \ref{b-wd} $\s_{[\lambda, \mu^+)}$ is a w-good $[\lambda, \mu^+)$-frame without the property that $\K_{[\lambda, \mu^+)}$ has no maximal models. Since $\s$ is a w-good $\lambda$-frame, $\K_\lambda$ has no maximal models. Therefore, by Lemma \ref{model-lambda} it follows that $\K_{[\lambda, \mu^+)}$ has no maximal models.
\end{proof}

With all the work we have done we can conclude the theorem promised at the beginning of the section.

\begin{theorem}\label{ext-ext}
Assume $\K$ is an AEC with the $[\lambda, \mu^+)$-amalgamation property. If $\s$ is a w-good $\lambda$-frame and $\K$ is $(\lambda, \mu)$-tame, then $\s_{[\lambda, \mu^+)}$  is a w-good $[\lambda, \mu^+)$-frame.
\end{theorem}
\begin{proof}
Follows from Lemma \ref{will-ext}, Lemma \ref{b-wd} and Lemma \ref{nmm}.
\end{proof}

In \cite[4.16]{vaseyd} Vasey weakens the hypothesis of the above theorem for good frames from $\K$ has the amalgamation property for that of $\K$ has the weak amalgamation. In the proof, it is crucial the density of basic types, therefore we do not know if one can weaken the hypothesis in the above theorem.

\section{Applications}

The following notation will be very useful in this section:

\begin{nota}
We denote by $( \ast )_{\lambda}$ the assertion ``$\Ii(\K,\lambda)= \Ii(\K,\lambda^+)=1\leq \Ii(\K,\lambda^{++}) < 2 ^{\lambda^{++}}$".
\end{nota}

In this section we will show how w-good frames can be used to prove the following:

\begin{theorem}\label{main}\footnote{As mentioned in the introduction, Shelah claims the same conclusion from fewer assumptions (see Fact \ref{sh1} and the two paragraphs above  it). 
}
Suppose $2^{\lambda}< 2^{\lambda^{+}} < 2^{\lambda^{++}}$ and $2^{\lambda^{+}} > \lambda^{++}$. If $( \ast )_{\lambda}$ and $\K$ is $(\lambda, \lambda^+)$-tame then $\K_{\lambda^{+++}} \neq \emptyset$.
\end{theorem}

The proof presented here follows the blueprint displayed in \cite{sh576}, unless otherwise noted all the definitions in this section were introduced by Shelah in \cite{sh576}. We would like to point out that most of what we prove here is already proved by Shelah in \cite{sh576}, but we decided to write down the proofs since some of Shelah's proofs are obscure, in particular those of Section 4.3, and they are central in the study of AECs.

 The proof of Theorem \ref{main} is done by contradiction. We will assume that $\K_{\lambda^{+++}} = \emptyset$ and using this property we will construct an explicit w-good $\lambda$-frame. Then using tameness together with Theorem \ref{ext-ext} we will get a contradiction by building a model of size $\lambda^{+++}$.

\subsection{Definition and basic properties} The next definition is crucial.

\begin{defin}
$(a, M_0, N_0) \in K_{\lambda}^{3}$ is \emph{minimal} when: if $(a, M_0, N_0) \leq_{h_l} (a_l, M_1, N_1^{l})$ for $l \in \{ 1, 2\}$ and $h_1\upharpoonright_{M_0}= h_2\upharpoonright_{M_0}$ then $\gtp(a_1/ M_1, N_1^1)=\gtp(a_2/ M_1, N_1^2)$.

A type $p \in \gS(M)$ is \emph{minimal} for $M \in \K_{\lambda}$, if for some $a$ and $N\in \K_\lambda$ we have that $(a, M, N) \in K_{\lambda}^{3}$ is minimal and $p=\gtp(a/M, N)$.
\end{defin}

With this definition we are ready to introduce our candidate for the w-good $\lambda$-frame. This frame was introduced in \cite[\S VI.8.3]{shelahaecbook}.\footnote{In  \cite[\S VI.8.3]{shelahaecbook} Shelah shows, under the hypothesis of Fact \ref{sh1}, that $\s_{min}$ is an almost good $\lambda$-frame. The reason we only show that $\s_{min}$ is a w-good $\lambda$-frame is because by Section 3 this is enough to get a model of size $\lambda^{+++}$ and because the known proofs of the other properties use the machinery of \cite[\S VII]{shelahaecbook} which we avoid.}

\begin{defin}\label{min-frame}
We define  $\s_{min}=(\K_{min}, \dnf_{min}, \gS^{bs}_{min})$ as follows:
\begin{itemize}
\item $\K_{min}= \K_{\lambda}$.
\item $\gS^{bs}_{min}=\{ \gtp(a/M, N) : (a, M, N) \in K_{\lambda}^{3} \text{\; minimal} \}$.
\item Given $M_{0} \lea M_{1} \lea N$ and $a\in |N| \backslash |M_{1}|$ we define
$a \dnf_{M_0}^{N} M_1$ if and only if $\gtp(a/M_1, N)\upharpoonright_{M_0}$ is minimal. 
\end{itemize}
\end{defin}

An easy consequence of  Fact \ref{AP} is the following.

\begin{remark}\label{AP2}
Suppose $2^{\lambda}< 2 ^{\lambda^{+}} < 2 ^{\lambda^{++}}$. Let $\K$ be an AEC.  If $(\ast)_\lambda$ then $\K_{\{\lambda, \lambda^{+}\}}$ has the amalgamation property. 
\end{remark}

Therefore the theorems on this section that assume that $\K$ has $\lambda$ or $\lambda^+$ amalgamation follow from the hypothesis of Theorem \ref{main}.

\begin{defin}\
\begin{enumerate}
\item $(a_0, M_0, N_0) \in K_{\lambda}^{3}$ has the \emph{weak extension property} if there is $(a_1, M_1, N_1) \in K_{\lambda}^{3}$ such that $(a_0, M_0, N_0) < (a_1, M_1, N_1) $.
\item $K_{\lambda}^{3}$ has \emph{no maximal pre-type} if every $(a_0, M_0, N_0) \in K_{\lambda}^{3}$ has the weak extension property.
\end{enumerate}
\end{defin}

As one can see from the definition of minimal pre-type, a pre-type can be minimal if there is no pre-type above it, but we will show that under the hypothesis of Theorem \ref{main} this can not happen.
This appears first as \cite[2.4]{sh576}, but a more straightforward proof is given in \cite[7.11]{grossberg2002} (in \cite{grossberg2002} it is assumed that the class is a PC class, but the hypothesis is not necessary).

\begin{fact}\label{wep}
Let $\K$ be an AEC. If $\Ii(\K, \lambda)= \Ii(\K, \lambda^+)=1$ and $\K_{\lambda^{++}}\neq \emptyset$ then  $K_{\lambda}^{3}$ has no maximal pre-type.
\end{fact} 

Now that we have that out of the way, we will show some basic properties about minimal pre-types. The following is \cite[2.6]{sh576}. Although the proofs are easy, we sketch them since they don't appear on \cite{sh576} and this facts are used throughout the paper.

\begin{lemma}\label{extendminimal} 
\begin{enumerate}\
\item If $(a, M_0, N_0) \leq (a, M_1, N_1) \in K^3_{\lambda}$ and $(a, M_0, N_0)$ is minimal then $(a, M_1, N_1)$ is minimal.
\item  ($\lambda$-amalgamation property is used) $(a, M_0, N_0)$ is minimal  if and only if the following holds: if $(a, M_0, N_0) \leq_{h_l} (a_l, M_1, N_1)$ for $l \in \{ 1, 2\}$ and $h_1\upharpoonright_{M_0}= h_2\upharpoonright_{M_0}$ then $\gtp(a_1/ M_1, N_1)= \gtp(a_2/ M_1, N_1)$.
\item  ($\lambda$-amalgamation property is used) If $(a, M_0, N_0) \in K_{\lambda}^{3}$, $p=\gtp(a/M_0, N_0)$ and $p$ is minimal then  $(a, M_0, N_0)$ is minimal.
\item ($\lambda$-amalgamation property is used) Let $M \lea M' \in \K_{\lambda}$. If $p \in \gS(M)$ minimal and $q \in \gS(M')$ extending $p$  then $q$ is a minimal type.
\end{enumerate}
\end{lemma}
\begin{proof}
(1), (3) and (4) are straightforward so let us sketch (2). The forward direction is trivial so let us show the backward one. Suppose $(a, M_0, N_0) \leq_{h_l} (a_l, M_1, N_1^{l})$ for $l \in \{ 1, 2\}$ and $h_1\upharpoonright_{M_0}= h_2\upharpoonright_{M_0}$, then apply the amalgamation property such that the following diagram commutes:
\[
 \xymatrix{\ar @{} [dr] N_2  \ar[r]^{j}  & N \\
N_0 \ar[u]^{h_2} \ar[r]^{h_1} & N_1  \ar[u]_{\id}
 }
\]

Then simply apply the hypothesis to $h_1'= h_1$, $h_2' = j \circ h_2$ and $N$.

\end{proof}

First let us show that $\s_{\min}$ is a pre-$\lambda$-frame. This appears without a proof in \cite[\S VI.8.1(1)]{shelahaecbook}.
\begin{lemma}\label{pre-frame}
Suppose $2^{\lambda}< 2^{\lambda^{+}}$. If $\K$ is $\lambda$-categorical and $1\leq \Ii( \K, \lambda^{+}) < 2 ^{\lambda^{+}}$ then $\s_{min}=(\K_{min}, \dnf_{min}, \gS^{bs}_{min})$ is a pre-$\lambda$-frame.
\end{lemma}
\begin{proof}
It is clear that (1) through (3) of the definition of pre-$\lambda$-frame are satisfied, so let us check that (4) through (6) are satisfied:
\begin{enumerate}
\skipitems{3}
\item \underline{Invariance:} It follows from the fact that minimal pre-types are closed under isomorphisms.
\item \underline{Monotonicity:} It follows from Lemma \ref{extendminimal}(4).
\item  \underline{Non-forking types are basic:} By definition.
\end{enumerate}
\end{proof}

Moreover, we can show the following.

\begin{lemma}\label{pre-frame2}
Suppose $2^{\lambda}< 2^{\lambda^{+}}$. If $\K$ is $\lambda$-categorical and $1 \leq \Ii (\K, \lambda^{+}) < 2 ^{\lambda^{+}}$  then $\s_{min}$ satisfies:
\begin{enumerate}
\skipitems{1}
\item $\K_\lambda$ has amalgamation, joint embedding and no maximal models.
\skipitems{3}
\item Uniqueness.
\skipitems{2}
\item Continuity.
\end{enumerate}
\end{lemma}
\begin{proof}\
\begin{enumerate}
\skipitems{1}
\item \underline{$\K_\lambda$ has amalgamation, joint embedding and no maximal models:} The amalgamation property follows from Remark \ref{AP2}. Joint embedding follows from $\lambda$-categoricity and no maximal models from $\lambda$-categoricity and the fact that $\K_{\lambda^+} \neq \emptyset$.
\skipitems{3}
\item \underline{Uniqueness:} It follows from the definition of minimal type.
\skipitems{2}
\item \underline{Continuity:} Let $\delta < \lambda^+$, $\{ M_i : i < \delta \} \subseteq \K_{\lambda}$  an increasing continuous chain, $\{p_{i} : i < \delta \}$ with $p_i \in \gS^{bs}_{min}(M_i)$ and for $i < j < \delta$ implies that $p_i = p_j \upharpoonright _ {M_i}$ and $p \in \gS^{na} (M_\delta)$ an upper bound. Since $p\upp_{M_0}= p_0$ and $p_0$ is minimal by Lemma \ref{extendminimal}(4) it follows that $p$ is minimal and hence basic.

Moreover, if each $p_i$ does not fork over $M_0$, then by definition $p\upp_{M_0}$ is minimal. Hence $p$ does not fork over $M_0$.

\end{enumerate}
\end{proof}

Therefore to show that $\s_{min}$ is a w-good $\lambda$-frame, we just need to show that it satisfies weak density and existence of nonforking extension. The proofs of these two facts are more complicated and will use all the hypothesis of Theorem \ref{main} together with the assumption that $\K_{\lambda^{+++}} = \emptyset$. Before we do that there is a very useful property that we get by assuming that $\K_{\lambda^{+++}} = \emptyset$.

\begin{defin}
Let $M \in \K_{\mu}$ and $LS(\K) \leq \lambda \leq \mu$ infinite cardinals. $M$ is \emph{universal above $\lambda$} if and only if for all $N_{0}, N_{1} \in \K_{[\lambda, \mu]}$ such that $N_{1} \gea N_{0} \lea M$ there is $f: N_1 \rightarrow_{N_{0}} M$. 
\end{defin}

 The following is similar to \cite[2.2]{sh576}, but instead of working in $\lambda^{++}$ we work in $\lambda^{+++}$.

\begin{lemma}\label{monster}
  If $\K_{ \{ \lambda, \lambda^{+} \}}$ has the amalgamation property, $\K_{\lambda^{++}} \neq \emptyset$ and $\K_{\lambda^{+++}}  = \emptyset$, then there is $\ce \in \K_{\lambda^{++}}$ universal above $\lambda$. Moreover if $\Ii(\K, \lambda)=\Ii(\K, \lambda^+)=1$, for each $N \in \K_{\{ \lambda, \lambda^+\}}$ there is $\ce \in \K_{ \lambda^{++}}$ universal above $\lambda$ such that $N \lea \ce$.

\end{lemma}
\begin{proof}

Since $\K_{\lambda^{+++}} = \emptyset$ there is $\ce \in \K_{\lambda^{++}}$ maximal. We claim that $\ce$ is universal above $\lambda$.  Let $N_{0} \lea N_1 \in \K_{\{\lambda, \lambda^+\}}$, then since $\K_{ \{ \lambda, \lambda^{+} \}}$ has the amalgamation property, there is $M \in \K_{\lambda^{++}}$ such that the following diagram commutes:

\[
 \xymatrix{\ar @{} [dr] N_1  \ar[r]^{f}  & M\\
N_0 \ar[u]^{\id} \ar[r]^{\id} & \ce \ar[u]_{\id}
 }
\]

Since $\ce$ is maximal, we have that $\ce = M$. Hence $f: N_1 \rightarrow_{N_0} \ce$. The moreover part follows from $\lambda$-categoricity or  $\lambda^+$-categoricity copying $\ce$. 

\end{proof}

\subsection{Weak density}
The only place where we use the extra cardinal arithmetic hypothesis that $2^{\lambda^+} > \lambda^{++}$ is to prove the following lemma, since we are already assuming that $2^{\lambda}<2^{\lambda^+}$ this is a weak hypothesis.

The lemma below is \cite[2.7]{sh576}. Shelah's proof and our proof are very similar, but we have decided to include it for the sake of completeness.

\begin{lemma}\label{minimal}
Suppose $2^{\lambda} < 2^{\lambda^+} < 2^{\lambda^++}$ and $2^{\lambda^+} > \lambda^{++}$. If $(\ast)_\lambda$ and $\K_{\lambda^{+++}} = \emptyset$ then minimal pre-types are dense in $K_{\lambda}^{3}$, i.e., for every pre-type there is a minimal one above it. 
\end{lemma}
\begin{proof}
We do the proof by contradiction. Let $(a, M, N) \in K_{\lambda}^{3}$ with no minimal pre-type above it. We will build $\{ (a_{\eta}, M_{\eta}, N_{\eta}) : \eta \in 2^{<\lambda^+} \}$ and $\{ h_{\eta, \nu} : \eta < \nu \text{\, with \,} \eta, \nu \in 2^{<\lambda^+} \}$ by induction such that:
\begin{enumerate}
\item $(a_{<>}, M_{<>}, N_{<>}):= (a, M, N) $.
\item $(a_{\eta}, M_{\eta}, N_{\eta}) \in K_{\lambda}^{3}$ for all $\eta \in 2^{<\lambda^+}$.
\item If $\eta < \nu$ then $(a_{\eta}, M_{\eta}, N_{\eta}) \leq_{ h_{\eta, \nu} } (a_{\nu}, M_{\nu}, N_{\nu})$.
\item If $\eta < \nu < \rho$ then $h_{\eta, \rho} = h_{\nu, \rho} \circ h_{\eta, \nu}$.
\item $M_{\eta^{\wedge}0}=M_{\eta^{\wedge}1} , N_{\eta^{\wedge}0}=N_{\eta^{\wedge}1}, h_{\eta, \eta^{\wedge}0}\upharpoonright_{M_{\eta}}= h_{\eta, \eta^{\wedge}1}\upharpoonright_{M_{\eta}}$ for all $\eta \in 2^{<\lambda^+}$.
\item $\gtp(a_{\eta^{\wedge}0}/ M_{\eta^{\wedge}0}, N_{\eta^{\wedge}0}) \neq \gtp(a_{\eta^{\wedge}1}/ M_{\eta^{\wedge}1}, N_{\eta^{\wedge}1})$.
\item If $\eta \in 2^{\delta}$ and $\delta< \lambda^+$ limit then $(M_{\eta}, \{h_{\eta\upharpoonright_{\alpha}, \eta} \}_{\alpha<\delta}), (N_{\eta}, \{h_{\eta\upharpoonright_{\alpha}, \eta} \}_{\alpha<\delta})$ are the direct limits of $( \{ M_{\eta\upharpoonright_{\alpha}} : \alpha < \delta \}, \{ h_{\eta\upharpoonright_{\alpha},\eta\upharpoonright_{\beta} } : \alpha < \beta < \delta \} )$ and $( \{ N_{\eta\upharpoonright_{\alpha}} : \alpha < \delta \}, \{ h_{\eta\upharpoonright_{\alpha},\eta\upharpoonright_{\beta} } : \alpha < \beta < \delta \} )$ respectively where $a_{\eta}=h_{\eta\upharpoonright_{1}, \eta}(a)$. 
\end{enumerate}

\underline{Construction:} In the base step apply (1). On limits take the direct limits, so the only interesting case is when $\alpha = \beta +1$. By construction we are given $(a_{\eta}, M_{\eta}, N_{\eta})$, since $(a_{<>}, M_{<>}, N_{<>}) \leq_{ h_{<>, \eta} } (a_{\eta}, M_{\eta}, N_{\eta})$ it follows that  $(a_{\eta}, M_{\eta}, N_{\eta})$ is not minimal. Applying Lemma \ref{extendminimal} (2) we are done.

\underline{Enough:} By Lemma \ref{monster} there is $\ce \in \K_{\lambda^{++}}$ universal above $\lambda$. We build $\{ g_{\eta} : M_{\eta} \rightarrow \ce : \eta \in 2^{< \lambda^+} \}$ by induction such that:
\begin{enumerate}
\item For every $\nu < \eta$ then $g_{\nu} \circ h_{\eta, \nu} = g_{\eta}$.
\item $g_{\eta^{\wedge}0} = g_{\eta^{\wedge}1}$.
\end{enumerate}
\underline{Construction:} 
Base: Since $\K$ is $\lambda$-categorical there is $g_{<>}: M_{<>} \rightarrow \ce$.

Induction step: If $\alpha$ is limit using that $M_{\eta}$ is a direct limit and the fact that we are constructing a cocone we obtain $g_{\eta} : M_{\eta} \rightarrow \ce$.

If $\alpha = \beta + 1$. Suppose we have $g_{\eta} : M_{\eta} \rightarrow \ce$. By the first construction we have $h_{\eta, \eta^{\wedge}0}[M_{\eta}] \lea M_{\eta^{\wedge}0}$, copying back the structure we build $g$ and $M_{\eta^{\wedge}0}'$ such that:

\[
 \xymatrix{\ar @{} [dr] M_{\eta^{\wedge}0}'  \ar[r]^{\cong_{g}}  & M_{\eta^{\wedge}0} \\
M_{\eta} \ar[u]^{\id} \ar[r]^{\cong_{h_{\eta, \eta^{\wedge}0}}} & h_{\eta, \eta^{\wedge}0}[M_{\eta}] \ar[u]_{\id}
 }
\]

Then copying forward the structure with respect to $g_{\eta}$ we build $h$ and $M_{\eta^{\wedge}0}''$ such that:

\[
 \xymatrix{\ar @{} [dr] M_{\eta^{\wedge}0}'  \ar[r]^{\cong_{h}}  & M_{\eta^{\wedge}0}'' \\
M_{\eta} \ar[u]^{\id} \ar[r]^{\cong_{g_{\eta}}} & g_{\eta}[M_{\eta}] \ar[u]_{\id}
 }
\]

Then using the universality of $\ce$ we get $j: M_{\eta^{\wedge}0}'' \rightarrow_{g_{\eta}[M_{\eta}]} \ce$. So let $g_{\eta^{\wedge}0}:= j \circ h \circ g^{-1}$ and  $g_{\eta^{\wedge}1}:= j \circ h \circ g^{-1}$. Since $M_{\eta^{\wedge}0} = M_{\eta^{\wedge}1}$ it is well-defined.

\underline{Enough:}  For each $\eta \in 2^{\lambda^+}$ let $((a_{\eta}, M_{\eta}, N_{\eta}), \{h_{\nu, \eta} : \nu < \eta\})$ be the direct limit of $( \{ M_{\eta\upharpoonright_{\alpha}} : \alpha < \lambda^+ \}, \{ h_{\eta\upharpoonright_{\alpha},\eta\upharpoonright_{\beta} } : \alpha < \beta < \lambda^+ \} )$ and $( \{ N_{\eta\upharpoonright_{\alpha}} : \alpha < \lambda^+ \}, \{ h_{\eta\upharpoonright_{\alpha},\eta\upharpoonright_{\beta} } : \alpha < \beta < \lambda^+ \} )$ .

By the construction of $\{ g_{\nu} : \nu < \eta \}$ and the definition of direct limits there is $f_{\eta} : M_{\eta} \rightarrow \ce$ such that for any $\nu < \eta(f_{\eta} \circ h_{\nu, \eta} =g_{\nu})$. Using that $\ce$ is universal above $\lambda$ there is $f_{\eta}': N_{\eta} \rightarrow \ce$ such that $f_{\eta} \subseteq f_{\eta}'$.

Observe that for every $\eta \in 2^{\lambda^{+}}$ we have that $f_{\eta}'(a_{\eta}) \in \ce$. Since $\|\ce\|= \lambda^{++}$ and $2^{\lambda^{+}} > \lambda^{++}$ we have $\eta \neq \nu \in 2^{\lambda^+}$ such that $f_{\eta}'(a_{\eta}) = f_{\nu}'(a_{\nu})$. Let $\alpha < \lambda^+$ least such that $\eta\upharpoonright_{\alpha}= \nu\upharpoonright_{\alpha}$ and $\eta(\alpha) \neq \nu(\alpha)$, we may assume without loss of generality that $\eta(\alpha)=0$ and $\nu(\alpha) = 1$.

\fbox{Claim}  $\gtp(a_{\eta\upharpoonright_{\alpha} ^{\wedge}0}/ M_{\eta\upharpoonright_{\alpha} ^{\wedge}0}, N_{\eta\upharpoonright_{\alpha} ^{\wedge}0}) = \gtp(a_{\eta\upharpoonright_{\alpha} ^{\wedge}1}/ M_{\eta\upharpoonright_{\alpha} ^{\wedge}1}, N_{\eta\upharpoonright_{\alpha} ^{\wedge}1})$.

Observe that the following diagram commutes:

\[
 \xymatrix{\ar @{} [dr] N_{\eta \upharpoonright_{\alpha} ^{\wedge}0}  \ar[r]^{ f_{\eta}' \circ h_{\eta\upharpoonright_{\alpha}^{\wedge}0, \eta} }  & \ce \\
M_{\eta\upharpoonright_{\alpha}^{\wedge}0} \ar[u]^{\id} \ar[r]^{\id} & N_{\eta\upharpoonright_{\alpha}^{\wedge}1} \ar[u]_{f_{\nu}' \circ h_{\eta\upharpoonright_{\alpha}^{\wedge}1, \nu}}
 }
\]

Moreover,  since $f_{\eta}'(a_{\eta}) = f_{\nu}'(a_{\nu})$ we have that $f_{\eta}' \circ h_{\eta\upharpoonright_{\alpha}^{\wedge}0, \eta}(a_{\eta\upharpoonright_{\alpha} ^{\wedge}0}) = f_{\nu}' \circ h_{\eta\upharpoonright_{\alpha}^{\wedge}1, \nu}(a_{\eta\upharpoonright_{\alpha} ^{\wedge}1})$. $\dagger_{\text{Claim}}$

Finally observe that this contradicts (6) of the first construction.

\end{proof} 

From the above lemma, the assertion below follows trivially.

\begin{lemma}\label{weak-d}
Suppose $2^{\lambda}< 2^{\lambda^{+}} < 2^{\lambda^{++}}$ and $2^{\lambda^{+}} > \lambda^{++}$. If $(\ast)_\lambda$ and $\K_{\lambda^{+++}} = \emptyset$ then $\s_{min}$ has weak density.
\end{lemma}
\begin{proof}
Let $M \lta N$ both in $\K_\lambda$, then pick $a \in |N| \backslash |M|$. By the previous theorem there is $(a, M', N') \in K_{\lambda}^3$ such that $(a, M, N) \leq (a, M', N')$ and  $(a, M', N')$ is minimal. Hence $\gtp(a/ M', N') \in \gS_{min}^{bs}(M')$. 
\end{proof}

\subsection{Existence of nonforking extension}  Fact \ref{wep} asserts that $K^{3}_{\lambda}$ has the weak extension property, in this section we will deal with the extension property. 

\begin{defin}\
\begin{itemize}
\item $(a_0, M_0, N_0) \in K_{\lambda}^3$ has the \emph{extension property} if given $M_1\in \K_\lambda$ and $f: M_0 \rightarrow M_1$, there are $N_1 \in \K_\lambda$ and $g: N_0 \rightarrow N_1$ such that $(a_0, M_0, N_0) \leq_g (g(a_0), M_1, N_1)$ and $g \supseteq f$. 
\item $p \in \gS^{na}(M_0)$ has the \emph{extension property} if given $M_1 \in \K_\lambda$ such that $M_0 \lea M_1$ there is $q \in \gS^{na}(M_1)$ extending $p$.
\end{itemize}
\end{defin}

\begin{remark}
$p$ has the extension property if and only if there is $(a, M, N) \in K_{\lambda}^3$ such that $p=\gtp(a/ M, N)$ and $(a, M, N)$ has the extension property. 
\end{remark}

The following fact is \cite[2.11]{sh576}, to show it Shelah used the $\lambda$-amalgamation property. 

\begin{fact}\label{easyext}
If $(a, M_0, N_0) \leq (a, M_1, N_1) \in K_{\lambda}^{3}$ and $(a, M_1, N_1)$ has the extension property then $(a, M_0, N_0)$ has the extension property. 
\end{fact}

The proof of the following lemma is similar to \cite[2.9]{sh576}, but our proof is shorter since we assume $\lambda^+$-categoricity instead of $1\leq \Ii(\K, \lambda^+) < 2^{\lambda^+}$ and we assume that $\K_{\lambda^{+++}} = \emptyset$. 

\begin{lemma}\label{manyrealizations} Assume $\K_{ \{\lambda, \lambda^+ \} }$ has the amalgamation property, $\K$ is $\lambda^+$-categorical and $\K_{\lambda^{+++}} = \emptyset$. If $(a, M_0, N_0) \in K_{\lambda}^{3}$, $M_0 \lea R$ and $|\{ c \in R : c \text{ realizes } \gtp(a/M_0, N_0) \}| \geq \lambda^{+}$ then $(a, M_0, N_0)$ has the extension property.
\end{lemma}
\begin{proof}
We may assume $R \in \K_{\lambda^{+}}$ and by Lemma \ref{monster} there is $\ce \gea R$ universal above $\lambda$. Let $f: M_0 \rightarrow M_1$ , we may assume that $f=\id_{M_0}$. By universality there is $h: M_1 \rightarrow_{M_0} \ce$. Since $\| h[M_1] \| = \lambda$ then there is $c \in |R| \backslash  |h[M_1]|$ and $R'\lea R$ such that $(a, M_0, N_0) E_{at} (c, M_0, R')$. Then by definitions of $E_{at}$ and universality of $\ce$ there is $R'' \lea \ce$ and $g: N_0 \rightarrow_{M_0} R''$ such that $g(a)=c$.

Applying LST axiom to $h[M_1] \cup R''$ inside $\ce$ we get $S \lea \ce$. Let $S^{*} \gea M_1$ and $d: S^{*} \cong S$ such that $h \subseteq d$. Since $c\notin h[M_1]$ it follows that  $(d^{-1}(c), M_1, S^{*}) \in K_{\lambda}^3$  and one can show that $(a, M_0, N_0) \leq_{d^{-1} \circ g} (d^{-1}(c), M_1, S^{*})$ and $d^{-1} \circ g \supseteq \id_{M_0}$.

\end{proof}

The next step is to prove the extension property for minimal types, for that we use weak diamond principles. Weak diamonds were introduced (for $\lambda = \aleph_{0}$) by Devlin and Shelah in \cite{desh}.

\begin{defin}
Let $S\subseteq \lambda^+$ be a stationary set.  $\Phi_{\lambda^+}^{2}(S)$ holds if and only if $\forall F: (2^{\lambda}) ^{<\lambda^+} \rightarrow 2$  $\exists g: \lambda^+ \rightarrow 2$ such that $\forall f: \lambda^+ \rightarrow 2^{\lambda}$ the set $\{ \alpha \in S : F(f \upharpoonright_{\alpha}) = g(\alpha) \}$ is stationary. 
\end{defin}

The following facts will be used in the proof of Lemma \ref{diamond} and a proof of them can be found in \cite[\S 15]{ramibook}.
\begin{fact}
\
\begin{enumerate}
\item $2^{\lambda} < 2^{\lambda^+}$ if and only if $\Phi_{\lambda^+}^{2}(\lambda^+)$ holds.
\item $\Phi_{\lambda^+}^{2}(S)$ holds for a stationary set $S\subseteq \lambda^+$ if and only if $\forall F: (2 \times 2 \times \lambda^+) ^{<\lambda^+} \rightarrow 2$ $\exists g: \lambda^+ \rightarrow 2$ such that $\forall \eta \in 2^{\lambda^+} \forall \nu \in 2^{\lambda^+} \forall h: \lambda^+ \rightarrow \lambda^+$ the set $\{ \alpha \in S : F(\eta\upharpoonright_{\alpha} , \nu\upharpoonright_{\alpha}, h\upharpoonright_{\alpha}) = g(\alpha) \}$ is stationary.
\item If $\Phi_{\lambda^+}^{2}(\lambda^+)$ holds then there exists $\{ S_i \subseteq \lambda^+ : i < \lambda^+ \}$ pairwise disjoint stationary sets such that  $\Phi_{\lambda^+}^{2}(S_i)$ for each $i < \lambda^+$.
\end{enumerate}
\end{fact}

The lemma below is presented precisely in the way it will be used in the proof of Theorem \ref{ext}. It is similar to \cite[1.6(1)]{sh576}, but our assumptions and conclusions are weaker. 
\begin{lemma} \label{diamond}
Suppose $2^{\lambda} < 2^{\lambda^{+}}$. Let $\{ M_{\eta} : \eta \in 2^{\lambda^+} \}$ such that for each $\eta \in 2^{\lambda^+}$ :
\begin{enumerate}
\item $\{ M_{\eta\upharpoonright_{\alpha}} : \alpha < \lambda^+\}$ strictly increasing and continuous.
\item For all $\alpha < \lambda^+( M_{\eta\upharpoonright_{\alpha}} \in \K_\lambda)$. 
\end{enumerate}
If for every $\eta \in 2^{\lambda^+}$ and $\alpha < \lambda^+$ $M_{\eta\upharpoonright_{\alpha}^{\wedge}0}$ can not be embedded to $M_{\nu}$ over $M_{\eta\upharpoonright_{\alpha}}$ when $\eta\upharpoonright_{\alpha}^{\wedge}1 < \nu$ and $\nu \in 2^{<\lambda^+}$, then $\K$ is not $\lambda^{+}$-categorical.  
\end{lemma}
\begin{proof}
We may assume that for all $\nu \in 2^{<\lambda^+}$ $(|M_{\nu}|= \gamma_{\eta} \in \lambda^{+})$, for every $\eta \in 2^{\lambda^+} (\{ \gamma_{\eta\upharpoonright_{\alpha}} : \alpha < \lambda^+ \}$ is continuous) and in that case $\forall \eta \in 2^{\lambda^+}( |M_{\eta}|=\lambda^+)$.

For each $\delta \in \lambda^{+}$, $\eta \in 2^{\delta}$, $\nu \in 2^{\delta}$ and $h: \delta \rightarrow \delta$ define:

\begin{equation*}
F(\eta, \nu, h) = \begin{cases}
1 \quad \text{$|M_{\eta}|=|M_{\nu}|=\delta$ and $h: M_{\eta} \rightarrow M_{\nu}$ can be extended to an isomorphism from} \\
\qquad \text{$M_{\eta_{0}}$ to $M_{\bar{0}}$ where $\eta^\wedge 0< \eta_{0}$ and $\nu < \bar{0}$} \\

0 \quad \text{otherwise}
\end{cases}
\end{equation*}

Let $\{ S_i \subseteq \lambda^+ : i < \lambda^+ \}$ pairwise disjoint stationary sets such that  $\Phi_{\lambda^+}^{2}(S_i)$ for each $i < \lambda^+$, they exist by the previous fact.

By $\Phi_{\lambda^+}^{2}(S_i)$ for all $i < \lambda^+$ let $g_i: \lambda^{+} \rightarrow  2$ such that for any $\eta, \nu \in 2^{\lambda^+}$ and $h: \lambda^+ \rightarrow \lambda^+$ the following set is stationary:
\[ S_{i}^{*} = \{ \delta \in S_i : F(\eta\upharpoonright_{\delta}, \nu\upharpoonright_{\delta}, h\upharpoonright_{\delta})= g_i(\delta) \} \]

Now, given $X \subseteq \lambda^{+}$ we define $\eta_{X}: \lambda^+ \rightarrow 2$ as follows:
\begin{equation*}
\eta_{X}(\delta) = \begin{cases}
g_{i}(\delta) & \text{if $\exists i \in X(\delta \in S_i)$ } \\
0 &\text{otherwise}
\end{cases}
\end{equation*}

Observe that since $\{S_i : i <\lambda^+ \}$ are pairwise disjoint for each $X$ $\eta_{X}$is well-defined.

\fbox{Claim:} If $X \subseteq \lambda^+$ and $X \neq \emptyset$, then $M_{\eta_{X}} \ncong M_{\eta_{\emptyset}}$. 

Suppose $h: M_{\eta_{X}} \cong M_{\eta_{\emptyset}}$. Observe that $\eta_{\emptyset} = \bar{0}$. Let $i \in X$ and $S_{i}^{*}= \{ \delta \in S_i : F(\eta_{X}\upharpoonright_{\delta}, \bar{0}\upharpoonright_{\delta}, h\upharpoonright_{\delta})= g_i(\delta) \}$ be the stationary set obtained for $\eta_{X}, \bar{0}$ and $h$.

 Let $C_{\eta_X}=\{ \delta < \lambda^+ : |M_{\eta_X\upharpoonright_{\delta}}| = \delta \}$, $C_{\bar{0}}=\{ \delta < \lambda^+ : |M_{\bar{0}\upharpoonright_{\delta}}| = \delta \}$ and $D= \{ \delta < \lambda^+ : h\upharpoonright_{\delta}: \delta \rightarrow \delta\}$. Since they are all clubs we can pick $\delta \in C_{\eta_X} \cap C_{\bar{0}} \cap D \cap S_i^{*}$. Define $\eta := \eta_{X} \upharpoonright_{\delta}$ and $\nu=\bar{0}\upharpoonright_{\delta}$. There are two cases:

\begin{enumerate}
\item \underline{Case 1:} $\eta_X(\delta)=1$. Since $\delta \in S_{i}^{*}$  we have that $g_{i}(\delta) = F(\eta, \nu, h\upharpoonright_{\delta})$ and since $i \in X$ we have that  $\eta_X(\delta)=g_{i}(\delta)$. Hence $ F(\eta, \nu, h\upharpoonright_{\delta})=1$. Then by definition there is $g \supseteq h\upharpoonright_{\delta}$ and $\eta_{0} > \eta^{\wedge}0$ such that $g: M_{\eta_0}\cong M_{\bar{0}}$. 

By hypothesis $h: M_{\eta_X} \cong M_{\bar{0}}$, so consider $f := h^{-1} \circ g: M_{\eta^{\wedge}0} \rightarrow M_{\eta_{X}}$. Since $\{ M_{\eta_{X}\upharpoonright_{\alpha}} : \alpha < \lambda^+\}$ is strictly increasing and continuous there is $\alpha < \lambda^+$ such that $f[M_{\eta^{\wedge}0}] \subseteq M_{\eta_{X}\upharpoonright_{\alpha}}$, so $f: M_{\eta^{\wedge}0} \rightarrow M_{\eta_X\upharpoonright_{\alpha}}$. Moreover $\eta_X\upharpoonright_{\alpha} > \eta^{\wedge}1$ and $M_{\eta}$ is fixed under $f$; contradicting our hypothesis.

\item \underline{Case 2:} $\eta_X(\delta)=0$. Then observe that $h \supseteq h\upharpoonright_{\delta}: M_{\eta_x} \rightarrow M_{\bar{0}}$ where $\eta_X > \eta^{\wedge}0$. So $F(\eta, \nu, h\upharpoonright_{\delta})=1$. Since $\delta \in S_i^{*}$ it follows that $\eta_X(\delta)= 1$. A contradiction to the hypothesis of this case.  $\dagger_{\text{Claim}}$ 

\end{enumerate}

Therefore, $\K$ is not $\lambda^+$-categorical. 
\end{proof}

We recall one last definition before we tackle Theorem \ref{ext}.

\begin{defin}\
\begin{itemize}
\item Given $p=\gtp(a/M, N) \in \gS(M)$ and $f: M \cong R$ define $f(p) :=\gtp(f'(a)/ R, f'[N])$ such that $f': N \cong f'[N]$ and $f' \supseteq f$.
\item Let $p=\gtp(a/M, N) \in \gS(M)$ and $R \in \K_{\lambda}$ then $S_{p}(R):=\{f(p) : f: M \cong R \}$. 
\end{itemize}
\end{defin}

Observe that if $M$ and $R$ are not isomorphic then $S_{p}(R) = \emptyset$, but in this paper when we refer to this notion, we will always assume categoricity in $\lambda$. Hence it will always be not empty. We will use the following lemma.

\begin{lemma}\label{bijectionS-p}
Let $p=\gtp(a/M, N)\in \gS^{na}(M)$. If $M\cong_g R$  then $|S_{p}(M)| = |S_{p}(R)|$. 
\end{lemma}
\begin{proof}
Define $\Phi: S_p(M) \rightarrow S_p(R)$, by $\Phi(tp(f'(a)/M, f'[N]))=tp(g_{f}\circ f' (a)/R, g_{f} \circ f' [N])$ such that $f': N \cong f'[N]$ and $f' \supseteq f$ where $f: M \cong M$ and the following square commutes:

\[
 \xymatrix{\ar @{} [dr] f'[N]  \ar[r]^{\cong_{g_{f}}}  & g_f  \circ f' [N] \\
M \ar[u]^{\id} \ar[r]^{\cong_{g}} & R  \ar[u]_{\id}
 }
\]

It is easy to see that $\Phi$ is a bijection.
\end{proof}

The next theorem is \cite[2.13]{sh576}. Our proof is similar to that of Shelah, but we show that Lemma \ref{diamond} is enough. 

\begin{theorem}\label{ext}
Suppose $2^{\lambda}< 2^{\lambda^{+}} < 2^{\lambda^{++}}$ and $2^{\lambda^{+}} > \lambda^{++}$. Assume $(\ast)_\lambda$ and $\K_{\lambda^{+++}} = \emptyset$. If $(a, M, N)\in K_{\lambda}^3$ is minimal then it has the extension property.
\end{theorem}

Since the proof is very long we have divided it into three lemmas.

\begin{lemma}\label{ext-pp} Under the hypothesis of Theorem \ref{ext}.
Let $p=\gtp(a/M, N)\in \gS^{na}(M)$ such that $p$ does not have the extension property. If $q$ extends $p$ then $q$ has less than $\lambda^+$ realizations.
\end{lemma}
\begin{proof}
Follows from Lemma \ref{manyrealizations}.
\end{proof}

\begin{lemma} \label{int1} Under the hypothesis of Theorem \ref{ext}. Let $p=\gtp(a/M, N)\in \gS^{na}(M)$ such that $(a, M, N)$ is reduced, minimal and does not have the extension property. Then there is a reduced pre-type $(a, M' ,N' )  \geq (a, M, N)$ such that $|S_{tp(a/M', N')}(M')| \geq \lambda^{++}$.
\end{lemma}
\begin{proof}
Let $M \lea \ce$ such that $\ce\in \K_{\lambda^{++}}$ universal above $\lambda$, this exists by Lemma \ref{monster}. We do the proof by contradiction, so suppose it is not the case. We build $\{ (a, P_\alpha, Q_\alpha) : \alpha < \lambda^+ \}$ such that:
\begin{enumerate}
\item $(a, P_0, Q_0) := (a, M,N)$.
\item $\{ P_\alpha : \alpha < \lambda^+ \}$ and $\{ Q_\alpha : \alpha < \lambda^+ \}$ are increasing and continuous.
\item $P_\alpha \lta P_{\alpha +1}$.
\item $(a, P_{\alpha}, Q_{\alpha}) \in K_{\lambda}^{3}$ is reduced for each $\alpha < \lambda^+$.  
\end{enumerate} 
The construction of the chain is done by combining Fact \ref{wep} and Fact \ref{reduce}.

We also build $\{ R_\alpha : \alpha < \lambda^+ \}$ and $\{ \Gamma_\alpha : \alpha < \lambda^+ \}$ such that:
\begin{enumerate}
\item $R_0 := M$
\item $\forall \alpha < \lambda^+ (R_\alpha \lea \ce \text{ and } R_\alpha \in \K_\lambda)$.
\item  $\{ R_\alpha : \alpha < \lambda^+ \}$ is increasing and continuous.
\item $\Gamma_\alpha = \{ q_i^{\alpha} : i < \lambda^+ \}= \bigcup_{\gamma < \lambda^+} S_{\gtp(a/R_\gamma, Q_\gamma)}(R_\alpha)$.
\item  $\forall \beta < \lambda^+ \forall q ( q = q_i^{\alpha} \text{ for } \alpha, i < \beta \text{ then there is no } N' \in \K_{\lambda} \text{ such that } \ce \gea N' \gea R_{\beta +1} \text{ and } c \in |N'| \backslash |R_{\beta +1}| \text{ realizing } q )$. 
\end{enumerate} 
\underline{Construction:} If $\alpha = 0$ apply (1) and if $\alpha$ is limit one takes unions. So the only interesting case is when $\alpha = \beta +1$. By hypothesis given $\gamma < \lambda^{+}$ we have that $|S_{\gtp(a/P_\gamma, Q_\gamma)}(P_\gamma)| \leq \lambda^+$, then by $\lambda$-categoricity and Lemma \ref{bijectionS-p} $|S_{\gtp(a/P_\gamma, Q_\gamma)}(R_\beta)| \leq \lambda^+$. So let $\{ q_i^{\beta} : i < \lambda^+ \}$ an enumeration of $\bigcup_{\gamma < \lambda^+} S_{\gtp(a/R_\gamma, Q_\gamma)}(R_\beta)$. Let $\Sigma = \{ q_{i}^{\alpha} : i, \alpha < \beta \}$, clearly $|\Sigma| \leq \lambda$. Observe that by Lemma \ref{ext-pp} if $u \in \Sigma$ and $A_u = \{ c\in \ce : c \text{ realizes } u \}$ then $|A_u|\leq \lambda$. Hence $A= \bigcup_{ u \in \Sigma} A_u $ is of size $\lambda$ and let $R_{\beta + 1}$ the structure obtained by applying LST axiom to $A \cup R_\beta$ in $\ce$. $R_{\beta + 1}$ works.

\underline{Enough:} Let $P= \bigcup_{\alpha < \lambda^+} P_\alpha$ and $R = \bigcup_{\alpha < \lambda^+} R_\alpha$. By $\lambda^+$-categoricty there is $g: P\cong R$. Let $D = \{ \delta < \lambda^+ : g: P_\delta \cong R_\delta \}$, by continuity of the chains this is a club. Let $\delta \in D$ and $q = g(\gtp(a/ P_\delta, Q_\delta)) \in S_{\gtp(a/ P_\delta, Q_\delta)}(R_\delta)$, by the enumeration there is $i < \lambda^+$ such that $q = q_i^{\delta}$.

Let $\epsilon >\delta, i$. Since $\{ P_\alpha : \alpha < \lambda^+ \}$ is increasing and continuous there is $\gamma < \lambda^+$ such that $g^{-1}[R_{\epsilon + 1}] \lta P_{\gamma}$. Moreover since $(a, P_\gamma, Q_\gamma) \geq (a, P_\delta, Q_\delta)$, then $g(a) \in |g[Q_\gamma]| \backslash |R_{\epsilon + 1}|$ and realizes $q$. Hence $g[Q_\gamma]$ and $g(a)$ contradict (5).
\end{proof}

\begin{lemma}\label{v-pair} Under the hypothesis of Theorem \ref{ext}.
Let $p=\gtp(a/M, N)\in \gS^{na}(M)$ such that $(a, M, N)$ is reduced, minimal, does not have the extension property and $|S_{p}(M) | \geq \lambda^{++}$. If $R \in \K_{\lambda}$, $\Gamma\subseteq \bigcup \{ S_p(R') : R' \lea R, R' \in \K_\lambda \}$ and $|\Gamma| \leq \lambda^+$ then \[ \Gamma^* = \{ q \in S_p(R) : \exists R^* \in \K_{\lambda}( R \lea R^* \text{,}  R^* \text{ realizes } q \text{ and there is no \,} c \in |R^*| \backslash |R| \text{ realizing } u \in \Gamma)\} \]  
has size $\lambda^{++}$.
\end{lemma}
\begin{proof}
Let $R \lea \ce$ such that $\ce\in \K_{\lambda^{++}}$ universal above $\lambda$, this exists by Lemma \ref{monster}. Let $\{ C_\alpha : \alpha < \lambda^{++}\}$ an increasing and continuous resolution of $\ce$ such that $C_0 := R$.

Given $q \in S_p(R)$ there is $(a_q, R, T_q) \in K^{3}_{\lambda}$ such that $(a, M, N) \cong (a_{q}, R, T_q)$. Since $(a, M, N)$ is reduced it follows that $ (a_{q}, R, T_q)$ is reduced. Moreover, from the fact that $\ce$ is universal above $\lambda$ we may assume that $T_q \lea \ce$.

\fbox{Claim} 
\begin{enumerate}
\item If $q_1 \neq q_2 \in S_p(R)$ then $a_{q_1} \neq a_{q_2}$.
\item If $a_q \notin C_\alpha$ then $T_q \cap C_\alpha = R$.
\end{enumerate}
The proof of the first claim follows from the fact that $T_{q_1}, T_{q_2} \lea \ce$. As for the second claim, it is clear that $R \subseteq  T_{q} \cap C_\alpha$, so we will show the other inclusion. Let $b \in T_{q} \cap C_\alpha$, let $R'$ the structure obtained by applying  LST axiom to $\{b \} \cup R$  in $C_\alpha$ and let $T'$ the structure obtained by applying  LST axiom to $\{a_q \} \cup R'$  in $\ce$. Clearly $(a_q, R, T_{q}) \leq (a_q, R', T') \in K^{3}_\lambda$ and since $(a_q, R, T_{q})$ is reduced $T_{q} \cap R' = R$. Since $b \in T_{q} \cap R'$, it follows that $b \in R$. $\dagger_{\text{Claim}}$

For each $u \in \Gamma$, by definition $u \in S_p(R')$ for some $R' \lea R$. Hence by Lemma \ref{ext-pp} if $A_u = \{ c \in \ce : c \text{ realizes } u \}$,  it follows that $| A_u | \leq \lambda$. Since $| \Gamma | = \lambda^+$, $|A|= |\bigcup_{u \in \Gamma} A_u | \leq \lambda^+$.

Pick $\alpha < \lambda^{++}$ such that $A \subseteq C_\alpha$. Let

\[ \Sigma = \{ q : q=\gtp(a_q/R, T_{q}) \text{ and } a_q \notin C_\alpha \} ,\]
we will show that $\Sigma \subseteq \Gamma^*$ and $| \Sigma | \geq \lambda^{++}$.

Let $q \in \Sigma$, so $q = \gtp(a_q/ R, T_{q})$ for $a_q \notin C_\alpha$. Suppose there is $c \in |T_{q}| \backslash |R|$ and $u \in \Gamma$ such that $c$ realizes $u$. Since $T_{q} \lea \ce$, by definition $c\in A_u \subseteq C_{\alpha}$. Hence by claim (2) $c \in T_{q} \cap C_\alpha = R$, contradicting the fact that $ c \notin R$. 

Finally, since $| S_p(M)| \geq \lambda^{++}$, by $\lambda$-categoricity and Lemma \ref{bijectionS-p} we have that $|S_{p}(R)| \geq \lambda^{++}$. From the fact that $\| C_{\alpha} \| = \lambda^+$ and Claim (1), it follows that $| \Sigma| \geq \lambda^{++}$. 
\end{proof}

\begin{proof}[Proof of Theorem \ref{ext}]
Let $\ce\in \K_{\lambda^{++}}$ universal above $\lambda$, this exists by Lemma \ref{monster}.

We do the proof by contradiction, so assume that $p=\gtp(a/M, N)$ does not have the extension property.

By $\lambda$-categoricity there is $h: N \rightarrow \ce$ so we may assume that $N \lea \ce$.
Moreover by Fact \ref{reduce}, Lemma \ref{int1}, Lemma \ref{ext-pp} and Lemma \ref{easyext} we may assume that $(a, M, N)$ is reduced, minimal and $|S_{\gtp(a/M, N)} (M)|\geq \lambda^{++}$.

We build $\{ M_\eta : \eta \in 2^{<\lambda^+} \}$ and $\{ p_{\eta}^{l} : \eta \in 2^{<\lambda^+}, l \in \{ 0, 1\}\}$ such that:
\begin{enumerate}
\item $M_{<>}:= M$.
\item $\forall \eta \in 2^{<\lambda^+}( M_\eta \in \K_{\lambda} \text{ and } M_\eta \lea \ce)$.
\item If $\eta < \nu$ then $M_\eta \lta M_\nu$.
\item $\forall \eta \in 2^{<\lambda^+} \forall l \in \{0,1\}( p_\eta^l \in S_p(M_\eta))$.
\item $M_\eta$ realizes $p_{\eta\upharpoonright_{\beta}}^{l}$ if and only if $\beta < lg(\eta)$ and $\eta(\beta)=l$. 
\end{enumerate}
 
Before doing the construction let us show that this is enough. 

\underline{Enough:} Given $\eta \in 2^{\lambda^+}$ let $M_{\eta} = \bigcup_{\alpha < \lambda^+} M_{\eta\upharpoonright_{\alpha}}$. Realize that the construction above satisfies the hypothesis of Lemma \ref{diamond}. In particular, if $\eta \in 2^{\lambda^+}$ and $\alpha < \lambda^+ M_{\eta\upharpoonright_{\alpha}^{\wedge}0}$ can not be embedded to $M_{\nu}$ over $M_{\eta\upharpoonright_{\alpha}}$ if $\eta\upharpoonright_{\alpha}^{\wedge}1 < \nu$ and $\nu \in 2^{<\lambda^+}$ by condition (5). Hence by the conclusion of Lemma \ref{diamond} $\K$ is not $\lambda^+$-categorical, contradicting the hypothesis of the theorem.

\underline{Construction:} For the base step use (1) and in limit stages take unions. So the only interesting case is when $\alpha = \beta +1$. In that case, we are given by induction hypothesis $M_\eta$ and need to build $M_{\eta^{\wedge} 0}$, $M_{\eta^{\wedge} 1}$  and $p_\eta^{0}, p_\eta^1$. We build $\{ (N^{\eta}_{\delta}, a^{\eta}_\delta) : \delta < \lambda^{++} \}$ such that:
\begin{enumerate}
\item $( a^{\eta}_\delta, M_\eta, N^{\eta}_\delta) \in K_\lambda^3$ and $N^{\eta}_\delta \lea \ce$.
\item $\gtp( a^{\eta}_\delta / M_\eta,  N^{\eta}_\delta) \in S_p(M_\eta)$.
\item $N^\eta_\delta$ omits every $q \in \Gamma_\delta$, where $\Gamma_\delta = \{ p^l_{\eta \upharpoonright_{\beta}} : \beta < lg(\eta), l \neq \eta(\beta) \} \cup \{ \gtp(b/M_\eta, N^{\eta}_\gamma) : \gamma < \delta, b \in N^\eta_{\gamma} \text{ and } \gtp(b/ M_\eta, N^\eta_{\gamma}) \in S_p(M_\eta) \}$. 
\end{enumerate}

As before we will first show that this is enough and then we will do the construction.

\underline{Enough:} For every $\delta <\lambda^{++}$, let  $W^{\eta}_\delta = \{ \gamma < \lambda^{++} : \exists b \in N^{\eta}_{\delta}( \gtp(b/ M_\eta, N^{\eta}_\delta)= \gtp(a^\eta_\gamma/ M_\eta, N^{\eta}_\gamma) ) \} $. Observe that for every $\delta$ we have that $| W^{\eta}_\delta| \leq \lambda$ since $N^{\eta}_\delta$ realizes at most $\lambda$ types and by (3) if $\gamma \neq \gamma'$ then $\gtp(a^\eta_\gamma/ M_\eta, N^\eta_\gamma)\neq \gtp(a^\eta_{\gamma '}/ M_\eta, N^{\eta}_{\gamma '})$.

Therefore, there are $\delta < \epsilon < \lambda^{++}$ such that $\delta \notin W^{\eta}_\epsilon$ and $\epsilon \notin W^{\eta}_\delta$. Let $M_{\eta^{\wedge} 0} = N^{\eta}_\delta$, $M_{\eta^{\wedge} 1} = N^{\eta}_\epsilon$, $p_\eta^{0} = \gtp(a^{\eta}_\delta/ M_\eta,  N^{\eta}_\delta)$ and $p_\eta^{1} = \gtp(a^{\eta}_\epsilon/ M_\eta, N^{\eta}_\epsilon)$.   

By (3) they omit all the restrictions and by $\delta \notin W^\eta_\epsilon$ and $\epsilon \notin W^\eta_\delta$ it follows that $p_\eta^{0}$ is omitted in $M_{\eta^{\wedge} 1}$ and  $p_\eta^{1}$ is omitted in $M_{\eta^{\wedge} 0}$.

\underline{Construction:} If $\delta = 0$, then \[|\Gamma_0|=|  \{ p^l_{\eta \upharpoonright_{\beta}} : \beta < lg(\eta), l \neq \eta(\beta) \} | \leq \lambda^{+}. \]
Observe that $p$, $R= M_\eta$  and $\Gamma= \Gamma_0$ satisfy the hypothesis of Lemma \ref{v-pair}. Therefore, there is $(a^{\eta}_0, N^\eta_0)$ such that $M_\eta \lea N^{\eta}_0$, $\gtp(a^{\eta}_0/ M_\eta, N^{\eta}_0) \in S_p(M_\eta)$ and  no $c\in |N^{\eta}_0| \backslash |M_\eta|$ realizes a type in $\Gamma_0$. Moreover, since $M_\eta$ omits $\Gamma_0$ it follows that $N^{\eta}_0$  omits $\Gamma_0$ and since $\ce$ is universal above $\lambda$ we may assume that $N^{\eta}_0 \lea \ce$.

If $\delta$ is limit or successor, realize that $|\Gamma_\delta| \leq \lambda^{+}$, then apply Lemma \ref{v-pair} as we did in the base step.

This finishes the construction and since we got to a contradiction in the first enough statement, we conclude that $(a, M, N)$ has the extension property. 
\end{proof}

We are finally able to obtain that $\s_{min}$ satisfies the existence of nonforking extension property.

\begin{lemma}\label{ext-pr}
Suppose $2^{\lambda}< 2^{\lambda^{+}} < 2^{\lambda^{++}}$ and $2^{\lambda^{+}} > \lambda^{++}$. If $(\ast)_\lambda$ and $\K_{\lambda^{+++}} = \emptyset$ then $\s_{min}$ satisfies existence of nonforking extension property. 
\end{lemma}
\begin{proof}
 Let $p=\gtp(a/ M, N) \in \gS^{bs}_{min}(M)$ and $M \lea M'$. Since $p \in \gS^{bs}_{min}(M)$, there is $(a, M, N)$ minimal pre-type such that $p=\gtp(a/M, N)$. Then by Theorem \ref{ext} there are $g$ and $N'\in \K_\lambda$ such that $(a, M, N) \leq_{g} (b, M', N') \in K_{\lambda}^3$ and $g \supseteq \id$. Let $q=\gtp(b/ M', N')$, it easy to show that $g$ is the witness for $(a, M, N) E_{at} (b, M, N')$, so $p \leq q$. Since $q\upharpoonright_M=p$ is a minimal type, we conclude that $q$ does not fork over $M$. 
\end{proof}

\subsection{Conclusion}
Putting together everything we have done in this section we get:

\begin{theorem}\label{ext-frame++}
Suppose $2^{\lambda}< 2^{\lambda^{+}} < 2^{\lambda^{++}}$ and $2^{\lambda^{+}} > \lambda^{++}$. If $(\ast)_\lambda$ and $\K_{\lambda^{+++}} = \emptyset$ then there is a w-good $\lambda$-frame.
\end{theorem}
\begin{proof}
By Lemma \ref{pre-frame} $\s_{min}$ is a pre-frame. Then by Lemma \ref{pre-frame2} $\s_{min}$ satisfies everything except weak density and existence of nonforking extension. Finally, by Lemma \ref{weak-d} $\s_{min}$ satisfies weak density and by Lemma \ref{ext-pr} $\s_{min}$ satisfies existence of nonforking extension. 
\end{proof}

Now, using the ideas from Section 3 together with the above theorem we are able to prove Theorem \ref{main}. We repeat the statement of the theorem for the convenience of the reader.

\begin{customthm}{4.2}
Suppose $2^{\lambda}< 2^{\lambda^{+}} < 2^{\lambda^{++}}$ and $2^{\lambda^{+}} > \lambda^{++}$. If $\Ii(\K, \lambda) = \Ii(\K, \lambda^{+}) = 1 \leq \Ii(\K, \lambda^{++}) < 2^{\lambda^{++}}$and $\K$ is $(\lambda, \lambda^+)$-tame then $\K_{\lambda^{+++}} \neq \emptyset$.
\end{customthm}
\begin{proof}
Suppose for the sake of contradiction that $\K_{\lambda^{+++}}= \emptyset$. Then by Lemma \ref{ext-frame++} $\s_{min}$ is a w-good $\lambda$-frame. Since $\K$ is $(\lambda, \lambda^{+})$-tame and $\K_{\{ \lambda, \lambda^+\} }$ has the amalgamation property by Remark \ref{AP2}, it follows from Theorem \ref{ext-ext} that $\s_{min, \{ \lambda, \lambda^+ \}}$ is a w-good $[ \lambda, \lambda^{++} )$-frame. Hence by Theorem \ref{model} $\K_{ \lambda^{+++} }\neq \emptyset$, which contradicts the hypothesis. 
\end{proof}

Lastly, let us show how we can apply Theorem \ref{main} to universal classes. In \cite{sh300} Shelah introduced the concept of universal class in the non-elementary setting.

\begin{defin}\label{universal}
A class of structures $K$ is a \emph{universal class} if:
\begin{enumerate}
\item $K$ is a class of $\tau$-structures, for some fixed vocabulary $\tau = \tau(K)$.
\item $K$ is closed under isomorphisms.
\item $K$ is closed under $\subseteq$-increasing chains.
\item If $M \in K$ and $N \subseteq M$, then $N \in K$.  
\end{enumerate}
Observe that if $K$ is a universal class then $\K= (K, \subseteq)$  is an AEC with $\LS (\K) = |\tau (K)| + \aleph_0$. We identify $K$ and $\K$.
\end{defin}

When $\K$ is a universal class, without any additional hypothesis, Will Boney proved that $\K$ is $(< \aleph_{0})$-tame. It appears in \cite[3.7]{vaseyd}.

\begin{fact}\label{tame-univ}
If $\K$ is a universal class, then $\K$ is $(< \aleph_{0})$-tame. In particular, $\K$ is $\lambda$-tame for every $\lambda \geq LS(\K)$.
\end{fact}

Putting together this fact with Theorem \ref{main} we get the following.

\begin{theorem}\label{univ}\footnote{Similarly to Theorem \ref{main}, this is not the best known result for universal classes, stronger results are obtained in \cite{mv}.}
Suppose $2^{\lambda}< 2^{\lambda^{+}} < 2^{\lambda^{++}}$ and $2^{\lambda^{+}} > \lambda^{++}$. Assume $\K$ is a universal class. If $( \ast )_{\lambda}$ then $\K_{\lambda^{+++}} \neq \emptyset$.
\end{theorem}
\begin{proof}
By Fact \ref{tame-univ} $\K$ is $(\lambda, \lambda^+)$-tame and by Theorem \ref{main} $\K_{\lambda^{+++}}\neq \emptyset$.
\end{proof}

Observe that the proof of Theorem \ref{univ} is around 30 pages long (we cited a couple of facts in this paper), while Shelah's original proof is around 250 pages long, making the above proof for universal classes 200 pages shorter.  We use the additional hypothesis that $2^{\lambda^+} > \lambda^{++}$, but as mention in Section 4.2 this is a weak hypothesis.

%\bibliography{}{}
%\bibliographystyle{amsalpha}

\end{document}